\documentclass[a4paper,11pt]{article}
\usepackage[a4paper,left=3cm,right=3cm,top=3.5cm,bottom=3.5cm]{geometry}
\usepackage{amsmath,amssymb, amsthm, mathrsfs, dsfont}

\makeatletter
\def\@endtheorem{\endtrivlist}
\makeatother

\usepackage{url} 

\usepackage[all]{xy}
\usepackage{verbatim}
\usepackage[shortlabels]{enumitem} 
\usepackage{cite} 
\usepackage[colorlinks,
			linkcolor=reference,
            citecolor=citation,
            urlcolor=e-mail,
            backref]{hyperref}
\usepackage{xcolor}
\definecolor{e-mail}{rgb}{0,.40,.80}
\definecolor{reference}{rgb}{.20,.60,.22}
\definecolor{citation}{rgb}{0,.40,.80}

\usepackage{fancyhdr}
\numberwithin{equation}{section}

\date{}
\title{Free $(\mathds{Z}/p)^n$-complexes and $p$-DG modules}
\author{Jeremiah Heller and Marc Stephan}

\newcommand{\Z}{\mathds{Z}}
\newcommand{\N}{\mathds{N}}

\newcommand{\F}{\mathds{F}}

\newcommand{\D}{{\mathcal D}}

\newcommand{\id}{\mathrm{id}}

\newcommand{\lol}{\mathit{ll}}

\DeclareMathOperator{\Hom}{Hom}

\DeclareMathOperator{\im}{im}

\DeclareMathOperator{\rank}{rank}

\newcommand{\NCh}{N\text{-}\mathbf{Ch}}
\newcommand{\pCh}{p\text{-}\mathbf{Ch}}
\newcommand{\Perf}{\mathbf{Perf}}

\newcommand{\pDGMod}{p\text{-}\mathbf{DG}\text{-}{_A}\mathbf{Mod}}

\DeclareMathOperator{\Cone}{Cone}

\newtheoremstyle{slanted}
{} 
{} 
{\slshape} 
{} 
{\bfseries}
{.} 
{.5em} 
{}
\makeatletter
\theoremstyle{slanted}
\newtheorem*{rep@theorem}{\rep@title}
\newcommand{\newreptheorem}[2]{%
\newenvironment{rep#1}[1]{%
 \def\rep@title{#2 \ref{##1}}%
 \begin{rep@theorem}}%
 {\end{rep@theorem}}}
\makeatother

\theoremstyle{definition}
\newtheorem{defn}{Definition}[section]
\newtheorem{rem}[defn]{Remark}
\newtheorem{assumption}[defn]{Assumption}
\newtheorem{ex}[defn]{Example}
\theoremstyle{slanted}
\newtheorem{cor}[defn]{Corollary}
\newtheorem{lem}[defn]{Lemma}
\newtheorem{thm}[defn]{Theorem}
\newtheorem{prop}[defn]{Proposition}
\newtheorem{conj}[defn]{Conjecture}
\newtheorem{question}[defn]{Question}
\newtheorem{problem}[defn]{Problem}

\begin{document}

\maketitle{}

\begin{abstract}
We reformulate the problem of bounding the total rank of the  homology of perfect chain complexes over the group ring $\F_p[G]$ of an elementary abelian $p$-group $G$ in terms of commutative algebra. This extends  results of Carlsson for $p=2$ to all primes.
 As an intermediate step, we construct an embedding of the derived category of perfect chain complexes over $\F_p[G]$ into the derived category of $p$-DG modules over a polynomial ring.


    \paragraph{Mathematics Subject Classification 2010.}
    Primary:
    \href{http://www.ams.org/mathscinet/msc/msc2010.html?t=55Mxx&btn=Current}{55M35},
    \href{http://www.ams.org/mathscinet/msc/msc2010.html?t=13Dxx&btn=Current}{13D09}.
    Secondary:
    \href{http://www.ams.org/mathscinet/msc/msc2010.html?t=13Dxx&btn=Current}{13D02}.
\end{abstract}

\section{Introduction}

Let $p$ be a prime, $G=(\Z/p)^n$  an elementary abelian $p$-group of rank $n$, and  $\F_p[G]$ the corresponding group ring over the field with $p$ elements. The following basic problem has a long and rich history, leading via the study of free $G$-actions on products of spheres back to the topological spherical space form problem.

\begin{problem}\label{problem}
Determine a lower bound for the total rank of the homology of non-acyclic perfect chain complexes of $\F_{p}[G]$-modules. 
\end{problem}
Recall that a perfect chain complex is a bounded complex of finitely generated projectives. Projective and free $\F_p[G]$-modules coincide and so a perfect chain complex of $\F_p[G]$-modules is a bounded complex of finitely generated free modules.

The following was conjectured to be a general answer to Problem \ref{problem}. Recently, Iyengar and Walker \cite{iyengarwalker} constructed counterexamples when $p$ is odd and $n\geq 8$. 
\begin{conj}\label{conjalg} 
Let $C$ be a non-acyclic perfect chain complex of $\F_p[G]$-modules. Then
\[\sum_{i}\dim_{\F_p} H_i(C) \geq 2^n.\]
\end{conj}

This conjecture arises as an algebraic version of Carlsson's conjecture in equivariant topology: the mod-$p$ homology of any non-empty, finite, free $G$-CW complex has total $\F_p$-dimension at least $2^n$. Indeed, one of the original motivations for phrasing Conjecture \ref{conjalg} was to reduce the topological conjecture to a purely algebraic problem.
Taking cellular chains on a $G$-CW complex links these two conjectures. Although the algebraic conjecture fails in general, 
Carlsson's conjecture remains open.

It bears mentioning that Conjecture \ref{conjalg} is known to hold for small rank $n$, namely for $n\leq 2$ if $p$ is odd and for $n\leq 3$ if $p=2$. For general $n$, the weaker bounds
\[
\sum_{i}\dim_{\F_p} H_i(C) \geq \begin{cases} n+1 &\mbox{if } p\geq 3 \\
2n &\mbox{if } p=2
\end{cases}
\]
hold. The best known bounds for general finite, free $G$-CW complexes in Carlsson's conjecture are the same.

In light of the counterexamples for odd primes and $n\geq 8$ to the bound $2^n$ in Conjecture \ref{conjalg}  and that the known bounds for the total dimension $\dim_{\F_p} H_\bullet(C)$ are stronger for $p=2$, 
the motivating questions of this paper arise: 
\begin{question}
 Can the bound $\sum_{i}\dim_{\F_p} H_i(C) \geq n+1$  for odd $p$ be improved?
Does Conjecture \ref{conjalg} hold for $2<n<8$? 
\end{question}

The results of this paper are a first step towards an improved bound for $p$ odd. 

To set the stage for our results we review what happens when $p=2$. The bounds for $p=2$ have been established by Carlsson in a series of papers \cite{carlsson83,carlsson86,carlsson87}, with an additional algebraic observation from \cite[p.~147, discussion following the proof of Corollary~3.6]{baumgartner}, \cite[Corollary~(1.4.21)]{alldaypuppe}. We distinguish three main steps in Carlsson's argument. 
\begin{enumerate}
\item Conjecture \ref{conjalg} is reformulated as a problem in commutative algebra via a Koszul duality argument. An equivalence is established between the derived categories of perfect chain complexes over $\F_2[G]$ and of free, finitely generated DG modules over the polynomial ring $\F_2[x_1,\ldots, x_n]$ with totally finite dimensional homology, where $G=(\Z/2)^{n}$.
\item The reformulation is realized as a problem of bounding sizes of {\it square-zero} matrices with entries in the polynomial ring $\overline{\F}_2[x_1,\ldots,x_n]$.
\item  Bounds for these reformulations are established.
\end{enumerate}

Our first main result below extends the second step to all primes. It contains a number $\ell_n$ that arises as the lower bound on the size of certain {\it $p$-nilpotent} matrices. More precisely, let $\ell_n$ be the minimum over all positive multiples $\ell$ of $p$ for which there exist integers $c_1,\ldots,c_\ell$ and a $p$-nilpotent $\ell\times \ell$-matrix $D=(f_{ij})$ whose entries are homogeneous polynomials $f_{ij}\in \overline{\F}_p[x_1,\ldots,x_n]$ of degree $c_i+1-c_j$ such that 
\begin{enumerate}[i)]
\item
$D$ is strictly upper triangular,
\item
$(f_{ij}(0))^{p-1}=0$, and
\item
$\rank(f_{ij}(x))= (p-1)\ell/p$ for all $x\in (\overline{\F}_p)^n\setminus\{0\}$.
\end{enumerate} 

\begin{thm}[Theorem \ref{thm:secondmaintext} for $\F_p$]
\label{thm:secondmain} Let $G=(\Z/p)^n$ and $\ell_n$ the integer defined above. Then $\dim_{\F_p} H_\bullet(C) \geq  2\ell_n/p$ for any non-acyclic, perfect chain complex $C$ over $\F_p[G]$.
\end{thm}

When $p=2$, this is a rephrasing of results of Carlsson; it is the second step mentioned above in the argument to establish bounds on $\dim_{\F_2} H_\bullet(C)$. A key element in the proof of Theorem \ref{thm:secondmain} is an extension of the first step above,  the reformulation of the $p=2$ case of Conjecture \ref{conjalg} as  a problem about DG modules over the polynomial ring, to all primes. 
To do this, we need to replace the differential of a chain complex or a DG module (which is a square-zero map) with a $p$-nilpotent map 
$d$. Of course, since $d^2\neq 0$,  we are no longer in the realm of  DG modules. Instead, we are working with objects called $p$-DG modules. 
Associated to a $p$-DG module $M$, are $p-1$ different homology groups: \[_sH_\bullet(M)=\ker(d^s)/\im (d^{p-s}),\] one for each $1\leq s\leq p-1$.
\begin{thm}[Theorem \ref{thm:firstmaintext} for $\F_p$]
\label{thm:firstmain} Let $G=(\Z/p)^n$. Let $A$ denote the polynomial ring $\F_p[x_1,\ldots, x_n]$ graded by $\deg(x_i)$=-1. Let $b_n$ be the minimum of
\[\sum_i \dim_{\F_p} {_1}H_i(M\otimes_A \F_p),
\]
where $M$ ranges over the finitely generated, free $p$-DG $A$-modules with $_sH_\bullet(M\otimes_A \F_p)\neq 0$ and $\dim_{\F_p} {_s}H_\bullet(M)<\infty$ for all $1\leq s \leq p-1$. Then
\[\sum_i \dim_{\F_p} H_i(C)\geq b_n
\]
for all non-acyclic, perfect chain complexes $C$ over $\F_p[G]$.
\end{thm}

Theorem \ref{thm:firstmain} is based on a construction of independent interest. For $G=(\Z/p)^{n}$, we provide an embedding of the derived category of perfect chain complexes over $\F_p[G]$ to the derived category of free, finitely generated $p$-DG modules over the polynomial ring in $n$ variables $A$ with totally finite dimensional homology (see section \ref{sec:functorbeta} and Theorem \ref{thm:embedding}). 
\begin{thm}
	There is an embedding of derived categories
	\[\D(\Perf(\F_p[G]))\to \D(\pDGMod).\]
\end{thm}
The embedding to free, finitely generated $p$-DG modules with totally finite dimensional homology extends Carlsson's Koszul duality construction from $p=2$ to all primes, but it is not an equivalence for $p>2$. 

To pass from $p$-DG modules to our main result, Theorem \ref{thm:secondmain}, we pick a basis for a free, finitely generated $p$-DG module $M$ and express its $p$-differential by a $p$-nilpotent matrix. We extend three additional results from $p=2$ to all primes to specialize to the matrices satisfying the three conditions in Theorem \ref{thm:secondmain}.

In Corollary \ref{cor:lowdimcases} of our main result, we recover immediately the known low-dimensional cases of Conjecture \ref{conjalg} for odd $p$. In future work, we intend to apply Theorem \ref{thm:secondmain} to investigate Conjecture \ref{conjalg} when $n=3$, the first open case for odd $p$.

\subsection*{Related work}
Let $C$ be a non-acyclic, perfect chain complex over $\F_p[G]$ for $G=(\Z/p)^n$. The known bound $n+1$ for the total rank of its homology $H_\bullet(C)$ was obtained by bounding the sum of Loewy lengths 
\begin{equation}\label{eq:loewylength}
\sum_i \lol_{\F_p[G]} H_i(C) \geq n+1
\end{equation}
together with the trivial observation $\dim_{\F_p} H_i(C)\geq \lol_{\F_p[G]} H_i(C)$. This is due to Carlsson for $p=2$, and to Allday, Baumgartner, Puppe  (see \cite[(1.4.14) Theorem]{alldaypuppe}) for odd $p$ when $C$ is the cellular cochain complex of a finite, free $G$-CW complex. In \cite{avramovbuchweitziyengarmiller}, Avramov, Buchweitz, Iyengar and Miller established the bound \eqref{eq:loewylength} for any $C$ and independently of the parity of $p$. They introduced a functor that sends $C$ to a DG module over a polynomial ring and established the bound \eqref{eq:loewylength} using levels of corresponding triangulated categories. Note that from the Loewy length alone it is not possible to improve the bound $n+1$ for the total rank of $H_\bullet(C)$.

For $p=2$, Carlsson bounded the total rank of $H_\bullet(C)$ after reformulating to DG modules over a polynomial ring. In \cite{avramovbuchweitziyengar}, Avramov, Buchweitz, and Iyengar established rank inequalities for differential modules over more general rings than the polynomial ring, subsuming Carlsson's bounds. It would be interesting to establish rank inequalities for $p$-differential modules.

Adem and Swan established Conjecture \ref{conjalg} for perfect chain complexes concentrated in two degrees in \cite[Corollary~2.1]{ademswan}.

Lastly, we mention that the functor $\beta$ defined in section \ref{sec:functorbeta} is related to constructions of Friedlander-Pevtsova \cite{Friedlander-Pevtsova} and Benson-Pevtsova \cite{Benson-Pevtsova}, associating 
 vector bundles 
 on $\mathds{P}^{n-1}$ to $\F_p[G]$-modules of constant Jordan type. If $M$ is a finitely generated $\F_p[G]$-module viewed as a $p$-complex concentrated in degree zero, then the differential $d$ of the $p$-differential graded $A$-module $\beta(M)$ identifies with  the map $\Theta_M$  used in the definition of the functors $\mathcal{F}_i$ from $\F_p[G]$-modules of constant Jordan type to vector bundles on $\mathds{P}^{n-1}$ in  \cite[Section 2]{Benson-Pevtsova}, see also \cite[Section 8.4]{BensonBook}.
More precisely, regrade $\beta(M)$ as a free, graded module over the polynomial ring $\mathbb{F}_p[x_1,\ldots, x_n]$ with its usual grading, $\deg(x_i)=1$, so that the $p$-differential $d$ of $\beta(M)$ raises the degree. With these grading conventions, the associated quasi-coherent module on $\mathds{P}^{n-1}$ is just $\widetilde{M}:=M\otimes_{\F_p}\mathcal{O}_{\mathds{P}^{n-1}}$ and the map $\widetilde{M}\to \widetilde{M}(1)$ associated to $d$ is exactly $\Theta_M$.

\subsection*{Outline}
In the same way that a DG module has an underlying chain complex, a $p$-DG module has an underlying $p$-complex. In sections \ref{sec:Ncomplexes} and \ref{sec:rootofunity}, we primarily recall the basics of $p$-complexes. Moreover, we show in Proposition \ref{prop:tensorpreserves} that tensoring with a $p$-complex preserves homotopies. In section \ref{sec:koszul}, we extend the notion of Koszul complex to a $p$-complex. We use this extension to connect perfect chain complexes over $\F_p[G]$ to $p$-DG modules with totally finite dimensional homology and to prove Theorem \ref{thm:firstmain} in section \ref{sec:chaincomplexestopDGmodules}. In section \ref{sec:matrices}, we establish three results, Proposition \ref{prop:minimalmodel}, Proposition \ref{prop:totallyfinite} and Theorem \ref{thm:compositionseries}, corresponding to the three conditions of the matrices in Theorem \ref{thm:secondmain}. We then prove our main result, Theorem \ref{thm:secondmain}, and recover the known low-dimensional cases of Conjecture \ref{conjalg} for odd $p$ from it. In section \ref{sec:embedding}, we prove that the construction from section \ref{sec:functorbeta} connecting perfect chain complexes over $\F_p[G]$ to $p$-DG modules induces an embedding on derived categories.

\subsection*{Acknowledgments}
It is our pleasure to thank Madhav Nori, Henrik R\"uping and Omar Antol\'in Camarena for helpful discussions. We would also like to thank two referees for useful comments and corrections.

Marc Stephan learned about the reformulation of Conjecture \ref{conjalg} for $p=2$ and the problem of examining the situation for odd primes in a talk of Gunnar Carlsson at the Algebraic Topology Open Problems/Work in Progress Seminar at MSRI in 2014.

Jeremiah Heller was supported by NSF Grant DMS-1710966. Marc Stephan was supported by SNSF grant 158932.

\section{$N$-complexes}\label{sec:Ncomplexes}
We will work with $p$-complexes instead of chain complexes. Instead of having a differential $d$ such that $d^2=0$, a $p$-complex is equipped with a map $d$ such that $d^p=0$. They have been introduced by Mayer in \cite{mayer}, who used them to define an alternative to the usual homology groups associated to a simplicial complex. However, as shown by Spanier \cite{spanier}, Mayer's new homology groups are expressible in terms of singular homology. As a result, it seems interest in $p$-complexes waned and they were forgotten about for the next fifty years, reappearing in work of Kapranov \cite{kapranov} and more recently in work of Khovanov and Qi \cite{khovanov, khovanovqi} on the categorification of small quantum groups. Despite Spanier's result, the derived categories of chain complexes and of $p$-complexes are different, as can be seen for example from Khovanov's computation of $K_0$ of the derived category of $p$-complexes \cite[Proposition 5]{khovanov}.
 
 We will be interested in the case when $p$ is a prime, but the general theory works for any integer $N\geq 2$. We recall the necessary basic definitions and properties following \cite{kapranov,kasselwambst,dubois-violettekerner}. We point readers interested in structural properties to \cite{iyamakatomiyachi} in which the homotopy category and derived category of $N$-complexes are equipped with triangulated structures.

Fix an integer $N\geq 2$ and a (unitary) ring $R$.

\begin{defn}
An \emph{$N$-complex} of $R$-modules is a $\mathbb{Z}$-graded $R$-module $C$ together with a homomorphism $d\colon C\to C$ of degree $-1$ such that $d^N=0$. We call $d$ the ($N$-)differential of $C$.

A \emph{morphism} $C\to C'$ of $N$-complexes is a homomorphism of degree $0$ that commutes with the differentials. We denote the category of $N$-complexes by $\NCh(R)$.

Two morphisms $f,g\colon (C,d)\to (C',d')$ are \emph{homotopic} if there exists a homomorphism $h\colon C\to C'$ of degree $N-1$ such that
\[
f-g= \sum_{i=0}^{N-1} (d')^{N-1-i}h d^i.
\]
\end{defn}

Note that a $2$-complex is just a chain complex and for $N=2$ the definitions specialize to the terminology used for chain complexes.

\begin{defn}
For $1\leq s \leq N-1$, the homology functor $_s H_\bullet$ from the category of $N$-complexes $\NCh(R)$ to the category of $\Z$-graded modules is defined by
\[
_s H_n(C)=\frac{\ker(d^s\colon C_n\to C_{n-s})}{\im(d^{N-s}\colon C_{n+N-s}\to C_n)}
\]
for an $N$-complex $C$ and integer $n$.

A morphism $f$ of $N$-complexes is called a \emph{quasi-isomorphism} if ${_s}H_\bullet(f)$ is an isomorphism for all $1\leq s\leq N-1$.
\end{defn}

As expected, homotopy equivalent $N$-complexes have isomorphic homology groups.
\begin{lem}[{\cite[Lemme 1.3]{kasselwambst}}]
Homotopic maps between $N$-complexes induce the same map on homology.
\end{lem}

\begin{defn}
An $N$-complex $C$ is \emph{acyclic} if $_sH_\bullet(C)=0$ for all $1\leq s\leq N-1$.
\end{defn}

Kapranov proved that it is enough to check acyclicity for one $s$.

\begin{prop}[{\cite[Proposition 1.5]{kapranov}}]\label{prop:acyclic}
An $N$-complex $C$ is acyclic if and only if $_sH_\bullet(C)=0$ for some $1\leq s\leq N-1$.
\end{prop}

Since a short exact sequence of chain complexes induces a long exact sequence in homology, we obtain the following result.
\begin{lem}[{\cite[Lemme 1.8]{kasselwambst}}]
\label{lem:longexact}
Let $1\leq s\leq N-1$. A short exact sequence
\[0\to C'\to C\to C''\to 0
\]
of $N$-complexes induces a long exact sequence
\[\xymatrix{_sH_\bullet(C')\ar[r] & _sH_\bullet(C)\ar[r] & _sH_\bullet(C'')\ar[d] \\
_{N-s}H_\bullet(C'')\ar[u] & _{N-s}H_\bullet(C)\ar[l] & _{N-s}H_\bullet(C'),\ar[l]}
\]
where the upward pointing arrow is of degree $-(N-s)$, the downward pointing arrow is of degree $-s$ and the horizontal arrows are induced by the maps in the short exact sequence.
\end{lem}

If $N\geq 3$, then the inclusion $i\colon \ker d^s\to \ker d^{s+1}$ for $1\leq s\leq N-2$ induces a natural map $i_*\colon _sH_\bullet(C)\to {_{s+1}H}_\bullet(C)$ of degree $0$, and the differential $d\colon \ker d^s\to \ker d^{s-1}$ for $2\leq s\leq N-1$ induces a natural map
$d_*\colon _sH_\bullet(C)\to {_{s-1}H}_\bullet(C)$ of degree $-1$.
\begin{lem}[{\cite[Lemme 1]{dubois-violettekerner}}]
\label{lem:les2}
For any integers $r,s\geq 1$ with $r+s\leq N-1$, the sequence
\[\xymatrix{_sH_\bullet(C)\ar[r]^{(i_*)^r} & _{s+r}H_\bullet(C)\ar[r]^{(d_*)^s} & _rH_\bullet(C)\ar[d]^{(i_*)^{N-s-r}} \\
_{N-r}H_\bullet(C)\ar[u]^{(d_*)^{N-r-s}} & _{N-s-r}H_\bullet(C)\ar[l]^{(i_*)^s} & _{N-s}H_\bullet(C)\ar[l]^{(d_*)^r}}
\]
is exact.
\end{lem}

Working over a field $k$, any chain complex is a direct sum of shifts of the contractible chain complex
\[\ldots \rightarrow 0\rightarrow k\stackrel{=}{\rightarrow} k\rightarrow 0 \rightarrow \ldots
\]
and of
\[\ldots \rightarrow 0\rightarrow k\rightarrow 0\rightarrow \ldots .
\]
For $N$-complexes, there are more non-contractible building blocks.
\begin{prop}[{\cite[Proof of Proposition~2]{tikaradze}}]\label{prop:structureNcomplexes} Let $k$ be a field. Any $N$-complex of $k$-vector spaces is a direct sum of shifts of the contractible $N$-complex
\[\ldots \rightarrow 0 \rightarrow k\stackrel{=}{\rightarrow} \ldots \stackrel{=}{\rightarrow} k \rightarrow 0\rightarrow \ldots\]
consisting of $N-1$ identity arrows $\id_k$ and of shifts of the $N$-complexes
\[\ldots \rightarrow 0 \rightarrow k\stackrel{=}{\rightarrow} \ldots \stackrel{=}{\rightarrow} k \rightarrow 0\rightarrow \ldots\]
with $i$ identity arrows $\id_k$, $0\leq i\leq N-2$.
\end{prop}
The proof is similar to establishing the normal form
\[M= \bigoplus_\alpha kv_\alpha \oplus kf(v_\alpha)\oplus \ldots\oplus kf^{i_\alpha}(v_\alpha)
\]
with $v_\alpha\in M$ and $i_\alpha\geq 0$ of a nilpotent endomorphism $f$ of a vector space $M$ over $k$.

\section{Primitive roots of unity and the tensor product for $N$-complexes}\label{sec:rootofunity}
Let $N\geq 2$ and suppose that $R$ is a commutative ring with unit $1\neq 0$. The sign $-1$ that appears in the differential of the tensor product of two chain complexes gets replaced by a primitive $N$th root of unity $q$ in the tensor product of $N$-complexes. Working with $N$-complexes, the combinatorics become more involved and require $q$-analogues of combinatorial identities. For the $p$-complexes we will consider, the ring $R$ will be an algebra over a field of characteristic $p$ and $q=1$ so that the ordinary combinatorial identities suffice.

Here we recall the general definitions and structural properties following \cite{kapranov,dubois-violette,bichon}. In addition, we prove that tensoring with an $N$-complex preserves homotopies (see Proposition \ref{prop:tensorpreserves}). 

\subsection{Primitive root of unity}
Fix an element $q$ of $R$.
\begin{defn}
Let $[\cdot]_q\colon \N\to R$ denote the function $[n]_q=\sum_{k=0}^{n-1} q^k$ with the convention that $[0]_q=0$.

The \emph{$q$-factorial} $[n]_q!$ of $n$ is defined by $[0]_q!=1$ and
\[[n]_q!=\prod_{k=1}^n[k]_q
\]
for integers $n\geq 1$.

For integers $n\geq m\geq 0$, the \emph{$q$-binomial coefficient} $\binom{n}{m}_q\in R$ is defined inductively by
\[\begin{cases}
\binom{n}{0}_q=\binom{n}{n}_q=1 \\
\binom{n+1}{m+1}_q=\binom{n}{m}_q + q^{m+1} \binom{n}{m+1}_q \quad \text{for } n-1\geq m\geq 0.
\end{cases}
\]
\end{defn}

\begin{ex}
If $R=\Z$ and $q=1$, then the map $[\cdot]_q\colon \N\to \Z$ is the inclusion, and the $q$-factorial and the $q$-binomial coefficients agree with the ordinary factorial and binomial coefficients, respectively.
\end{ex}

\begin{defn}
A \emph{distinguished primitive $N$th root of unity of $R$} is an element $q\in R$ such that $[N]_q=0$ and $[n]_q$ is invertible for all $1\leq n\leq N-1$.
\end{defn}

Note that $q^N=1$ as $q^N-1=(q-1)[N]_q=0$ for any distinguished primitive $N$th root of unity $q$. In particular distinguished primitive $N$th roots of unity are invertible.

\begin{ex}
If $R=\mathbb{C}$, then $q\in \mathbb{C}$ is a distinguished primitive $N$th root of unity if and only if $q$ is a primitive $N$th root of unity in the ordinary sense.

If $k$ is a field of characteristic $p>0$ and $N=p$, then $q=1$ is a distinguished primitive $p$th root of unity of the field $k$.

If $R$ is any commutative ring and $N=2$, then $q=-1$ is the unique distinguished primitive $N$th root of unity of $R$.
\end{ex}

The following identity holds by induction.
\begin{lem}\label{lem:binomcoeff} Let $q$ be a distinguished primitive $N$th root of unity of $R$. Then
\[\binom{n}{m}_q = \frac{[n]_q!}{[n-m]_q! [m]_q!}
\]
for all $N\geq n\geq m\geq 0$ with the convention that $\frac{[N]_q!}{[N]_q!}=1$.
\end{lem}

The case $q=1$, $R=\Z$ of the following identity is well-known to combinatorialists. We will need it to show that tensoring with an $N$-complex preserves homotopies.
\begin{lem}\label{lem:qidentity} Let $q$ be a distinguished primitive $N$th root of unity of $R$. For all $0\leq s\leq t<m$ the identity
\[q^{(s+1)t}\sum_{i=s}^{t} \binom{m-1-i}{m-1-t}_q\binom{i}{s}_q q^{-i(s+1)}= \binom{m}{m+s-t}_q
\]
holds.
\end{lem}
\begin{proof} If $s=t$, then the identity holds since
\[q^{(s+1)t}q^{-t(s+1)}=1=\binom{m}{m}_q.
\]
In particular, the identity holds for $m=1$. For $m>1$, the identity is established by double induction on $m$ and $t$. 
\end{proof}

\subsection{Tensor product}\label{sec:tensorproduct}
Suppose that $R$ is equipped with a distinguished primitive $N$th root of unity $q$.
\begin{defn}
The \emph{tensor product} $C\otimes D$ of two $N$-complexes $C$, $D$ is the $N$-complex given by the graded $R$-module $\{\oplus_{a+b=n} C_a\otimes D_b\}_{n\in \Z}$ with differential induced by
\[d(x\otimes y)=dx\otimes y + q^{-a} x\otimes dy
\]
for $x\in C_a$ and $y\in D_b$.
\end{defn}
The map $d$ satisfies $d^N=0$ as
\begin{equation}\label{eq:differentialtensor}
d^k(x\otimes y)=\sum_{m=0}^k q^{-a(k-m)} \binom{k}{m}_q d^m(x)\otimes d^{k-m}(y)
\end{equation}
by induction on $k$ and $\binom{N}{m}_q=0$ for $1\leq m\leq N-1$ by Lemma \ref{lem:binomcoeff}.

For any $N$-complex $C$, the functor $-\otimes C$ has a right adjoint $\Hom(C,-)$.

\begin{defn}
The \emph{Hom-complex} $\Hom(C,D)$ of two $N$-complexes $C$, $D$ is the $N$-complex with $\Hom(C,D)_n=\prod_{i\in \Z} \Hom_R(C_i,D_{i+n})$ for $n\in\Z$ and differential $d\colon \Hom(C,D)_n\to \Hom(C,D)_{n-1}$ defined by
\[d((f_i)_{i\in \Z})=(df_i-q^{-n}f_{i-1} d)_{i\in \Z}.
\]
The map $d$ satisfies $d^N=0$ as
\[d^k(f_i)_i= \left(\sum_{m=0}^k (-1)^{k-m}q^{(k-m)(k-m-1)/2}q^{-(k-m)n} \binom{k}{m}_q d^m f_{i-k+m}d^{k-m}\right)_i
\]
by induction on $k$ and since $\binom{N}{m}_q=0$ for $1\leq m\leq N-1$.
\end{defn}

The structural properties of the tensor product and the Hom-complex are summarized in the following result.
\begin{prop}[\cite{kapranov,dubois-violette,bichon}]
The category of $N$-complexes together with the tensor product and the $N$-complex given by $R$ concentrated in degree zero is a monoidal category. When $N\geq 3$, this monoidal structure does not admit a braiding in general. For any $N$-complex $C$, the functor $-\otimes C$ is left adjoint to the functor $\Hom(C,- )$.
\end{prop}

\begin{rem}\label{rem:symmetry} 
If $q=1$, then the monoidal structure is symmetric with symmetry induced by \[C\otimes D\cong D\otimes C, \quad x\otimes y\mapsto y\otimes x.\]
This holds in particular in our cases of interest, namely $p$-complexes in characteristic $p$. 
\end{rem}

We prove that tensoring with an $N$-complex preserves homotopies.
\begin{prop}\label{prop:tensorpreserves}
Let $h$ be a homotopy between two maps of $N$-complexes $f,g\colon C\to C'$. Then $h\otimes D\colon C\otimes D\to C'\otimes D$ is a homotopy between $f\otimes D$ and $g\otimes D$ for any $N$-complex $D$. Similarly, $D\otimes f$ and $D\otimes g$ are homotopic via
\[x\otimes y\mapsto q^{-a}x\otimes h(y)
\]
for homogeneous elements $x\in D_a$ and $y\in C_b$.
\end{prop}
\begin{proof}
To establish the first statement, we have to show that
\[\sum_{i=0}^{N-1} d_{C'\otimes D}^{N-1-i} (h\otimes D) d^i_{C\otimes D} (x\otimes y) = f(x)\otimes y - g(x)\otimes y
\]
for any homogeneous elements $x\in C_a$ and $y\in D_b$. While here we used subscripts to distinguish the differentials, we will henceforth denote most differentials by $d$.

Writing out $d_{C'\otimes D}^{N-1-i}$, $d^i_{C\otimes D}$ using \eqref{eq:differentialtensor}, and simplifying $q^N=1$, the left-hand side becomes
\[\sum_{\substack{0\leq i \leq N-1 \\ 0\leq j\leq N-1-i\\ 0\leq m\leq i}}
q^{-(a-m-1)(-1-i-j)}q^{-a(i-m)} \binom{N-1-i}{j}_q\binom{i}{m}_q (d^jhd^m x)\otimes (d^{N-1-j-m}y).
\]
Simplifying and changing the order of summation yields
\[\sum_{\substack{0\leq m \leq N-1 \\ 0\leq j\leq N-1-m\\ m\leq i\leq N-1-j}} q^{(a-m-1)(i+j+1)+a(m-i)} \binom{N-1-i}{j}_q\binom{i}{m}_q (d^jhd^m x)\otimes (d^{N-1-j-m}y).
\]

The sum over $0\leq m\leq N-1$ of the terms with $j=N-1-m$ and thus $i=m$ is 
\[\sum_{m=0}^{N-1} d^{N-1-m}hd^m x \otimes y=(f-g)(x)\otimes y.\]

We conclude showing that for $j\neq N-1-m$, the sum
\[\sum_{i=m}^{N-1-j} q^{(a-m-1)(i+j+1)+a(m-i)} \binom{N-1-i}{j}_q\binom{i}{m}_q
\]
is zero. Indeed, setting $s=m$, $t=N-1-j$ and applying Lemma \ref{lem:qidentity}, the sum becomes
\[q^{a(s-t)} \binom{N}{N+s-t}_q.
\]
This is zero by Lemma \ref{lem:binomcoeff} as $s<t$ by assumption on $j$.

The proof of the second statement is similar.  
\end{proof}

\section{An extension of Koszul complexes}\label{sec:koszul}
Recall that the Koszul complex $K_\bullet(x_1,\ldots, x_n)$ of the polynomial ring over a field $k$ with respect to the variables $x_1,\ldots,x_n$ is a minimal free resolution of the field. We extend the notion of Koszul complex to an $N$-complex $K_\bullet^N(x_1,\ldots,x_n)$. In general, the resulting $N$-complex will no longer be quasi-isomorphic to the $N$-complex given by the field $k$ concentrated in degree $0$, but the total $k$-dimension of its homology groups $_sH_\bullet$ will still be finite. This result is precisely what we will need in our extension of Carlsson's reformulation of Conjecture \ref{conjalg} to all primes.

Let $N\geq 2$ and let $R$ be a commutative ring with a distinguished primitive $N$th root of unity $q$.
\begin{defn}\label{def:koszulncomplex}
For an element $x\in R$, let $K^N_\bullet(x)$ denote the $N$-complex
\[\ldots \rightarrow 0\rightarrow R\stackrel{x}{\rightarrow} R\stackrel{x}{\rightarrow}\ldots \stackrel{x}{\rightarrow} R\rightarrow 0\rightarrow\ldots,\]
where the first $R$ lies in degree $N-1$ and the last $R$ lies in degree $0$.

For a sequence $x=(x_1,\ldots,x_n)$ of elements in $R$, define the $N$-complex $K^N_\bullet(x)$ by
\[K^N_\bullet(x)=K^N_\bullet(x_1)\otimes\ldots \otimes K^N_\bullet(x_n).
\]
\end{defn}

For $N=2$, the chain complex $K^2_\bullet(x)$ is the Koszul complex $K_\bullet(x)$ of the sequence $x$. In particular, if $x$ is a regular sequence in $R$, then $K^2_\bullet(x)$ is quasi-isomorphic to $R/(x)$ concentrated in degree $0$. This is not the case when $N>2$ as the following example shows already in one element $x$.

\begin{ex} Let $R=k[x]$ be the polynomial ring in one variable $x$ over a field $k$. Then $K_\bullet^N(x)$ is quasi-isomorphic to
\[\ldots \rightarrow 0\rightarrow k[x]/(x)\stackrel{x}{\rightarrow} k[x]/(x^2)\stackrel{x}{\rightarrow}\ldots \stackrel{x}{\rightarrow} k[x]/(x^{N-1})\rightarrow 0\rightarrow\ldots,\]
where $k[x]/(x)$ is in degree $N-2$ and thus $k[x]/(x^{N-1})$ is in degree $0$.
\end{ex}

We calculate the homology groups of $K^N_\bullet(x)$ for a sequence of more than one element in an example.
\begin{ex}
Let $N=p=3$ and let $R=\F_p[x,y]$ be the polynomial ring in two variables over $\F_p$ with $q=1$. Then the non-zero part of $K_\bullet^3(x,y)$ is
\[
R_{2,2}\stackrel{\begin{pmatrix} y \\ x\end{pmatrix}}{\longrightarrow} R_{2,1}\oplus R_{1,2} \stackrel{\begin{pmatrix}
 y & 0\\ x & y\\ 0 & x \end{pmatrix}}{\longrightarrow} R_{2,0}\oplus R_{1,1}\oplus R_{0,2} \stackrel{\begin{pmatrix}
 x & y & 0\\ 0 & x & y
 \end{pmatrix}}{\longrightarrow} R_{1,0}\oplus R_{0,1} \stackrel{\begin{pmatrix}
 x & y\end{pmatrix}}{\longrightarrow} R_{0,0}
\]
where $R_{i,j}=K^N_i(x)\otimes K^N_j(y)$. 

We calculate the homology groups $_sH_i:={_sH}_i(K^3_\bullet(x,y))$, $s=1,2$ and $i\in\N$, as $\F_p$-vector spaces.

For $s=1$, we obtain
\[\begin{cases} _1H_0=\F_p\oplus \F_p x\oplus \F_p y \\ _1H_1=\F_p(y,-x) \\ _1H_i= 0 \quad \text{for } i\geq 2.\end{cases}
\]

For $s=2$, we obtain
\[\begin{cases} _2H_0=\F_p \\ _2H_i= 0 \quad \text{for } i\geq 2\end{cases}
\]
by direct calculation, and $_2H_1\cong\F_p\oplus \F_p\oplus \F_p$ from the long exact sequence of Lemma \ref{lem:les2}. 
\end{ex}

While we will not use the following result, we record it because for $N=2$ it is a starting point for showing that the Koszul complex of a regular sequence is a free resolution. For an integer $l$, let $C[l]$ denote the shifted $N$-complex with $C[l]_i=C_{i-l}$.

\begin{lem}
Let $C$ be an $N$-complex and $x\in R$. For any $2\leq m\leq N$, there is a short exact sequence of $N$-complexes
\[0\longrightarrow C\longrightarrow C\otimes K^m_\bullet(x)\longrightarrow C\otimes (K^{m-1}_\bullet(x)[1])\longrightarrow 0
,\]
where the $m$-complex $K^m_\bullet(x)$ is considered as an $N$-complex and $K^1_\bullet(x)=R$.
\end{lem}
\begin{proof}
The desired short exact sequence is given in degree $i$ by
\[0\longrightarrow C_i\longrightarrow \bigoplus_{l=0}^{m-1} C_{i-l}\longrightarrow \bigoplus_{l=1}^{m-1} C_{i-l}\longrightarrow 0
.\]
\end{proof}

Our extension of Carlsson's reformulation of Conjecture \ref{conjalg} to all primes relies on the following proposition.
\begin{prop}\label{prop:NKoszulfinitelygen}
Let $N=p$ be a prime number, let $k$ be a field of characteristic $p$, and let $(R,q)$ be the polynomial ring $k[x_1,\ldots,x_n]$ in $n$ variables with distinguished primitive root of unity $1$. Then
\[\dim_{k} {_sH}_\bullet(K^N_\bullet(x_1,\ldots,x_n)) <\infty
\]
for all $1\leq s\leq N-1$.
\end{prop}
\begin{proof}
We will show the equivalent statement that each $x_i$ acts nilpotently on the homology ${_sH}_\bullet(K^N_\bullet(x_1,\ldots,x_n))$. In fact, we will show that the multiplication by $(x_i)^{p-1}$ on $K^N_\bullet(x_1,\ldots,x_n)$ is homotopic to the zero map. A nullhomotopy $h$ for $(x_i)^{p-1}\colon K^N_\bullet(x_i)\to K^N_\bullet(x_i)$ is given by $h(z)=z$ for $z$ in degree zero and $h(z)=0$ otherwise. Tensoring $h$ together with the identity maps on $K^N_\bullet(x_j)$ for $j\neq i$ yields the desired nullhomotopy for $(x_i)^{p-1}\colon K^N_\bullet(x_1,\ldots,x_n)\to K^N_\bullet(x_1,\ldots,x_n)$.
\end{proof}

\section{From chain complexes of $(\Z/p)^n$-modules to $p$-DG modules}\label{sec:chaincomplexestopDGmodules}

In this section, we prove our first main result, Theorem \ref{thm:firstmain} from the introduction, connecting chain complexes with a $(\Z/p)^n$-action  and $p$-DG modules over a polynomial ring via a functor $\beta$. We introduce terminology and basic properties of $p$-DG modules in section \ref{sec:NDGmodules}. We  construct $\beta$ and establish Theorem \ref{thm:firstmain} in section \ref{sec:functorbeta}.

\subsection{Definitions and basic properties of $p$-DG modules}\label{sec:NDGmodules}
Let $p$ be a prime number and let $k$ be a field of characteristic $p$. We fix $q=1$ as a distinguished primitive $p$th root of unity. We write $A=k[x_1,\ldots,x_n]$ for the polynomial ring in $n$ variables graded by $\deg(x_i)=-1$, $1\leq i \leq n$. Equipped with a trivial differential, we consider $A$ as a monoid in the category of $p$-complexes. The graded ring $A$ is not graded commutative in the ordinary sense, but it is a commutative monoid with the symmetry from Remark \ref{rem:symmetry}.

\begin{defn}
A $p$-differential graded $A$-module $M$ is a graded $A$-module $M$ together with an $A$-module map $d\colon M\to M$ (thus $d(am)=ad(m)$ for all $a\in A$ and $m\in M$) of degree $-1$ such that $d^p=0$.

The category of $p$-differential graded $A$-modules is denoted by $\pDGMod$.
\end{defn}

Any $p$-DG module has an underlying $p$-complex, i.e. $N$-complex with $N=p$, of $k$-modules. As for DG modules, a \emph{homotopy} between two maps of $p$-DG modules is a homotopy between the maps on underlying $p$-complexes that commutes with the $A$-action.

The category $\pDGMod$ fits into the Hopfological algebra framework of Khovanov \cite{khovanov} and Qi \cite{qi}. We will need it for the characterization of homotopy equivalences between $p$-DG modules via contractibility of mapping cones in Proposition \ref{prop:cone}. A reader willing to take this result for granted, may skip to the explicit description \eqref{eq:cone} of the mapping cone.

The category $\pDGMod$ can be understood as the category of graded modules over a certain comodule algebra $A_{\partial}$ over a Hopf algebra, as follows. We write $H = k[d]/(d^p)$.  This is a graded Hopf algebra, with $\deg(d)=-1$, comultiplication $\Delta(d) = d\otimes 1 + 1\otimes d$, antipode $S(d) = -d$ and counit $\varepsilon(d)=0$. As for $p$-complexes, we use the symmetry $x\otimes y\mapsto y\otimes x$ instead of $x\otimes y \mapsto (-1)^{\deg(x)\deg(y)} y\otimes x$ on the category of graded $k$-vector spaces. Thus $H$ is commutative and cocommutative. The category of graded (left) $H$-modules ${}_{H}{\mathbf{Mod}}$ is isomorphic to the category of $p$-complexes considered in section \ref{sec:Ncomplexes}. 

Now consider the algebra $A_{\partial}=A\otimes_{k}H$. This is a graded (right) $H$-comodule algebra with coaction $\Delta_\partial\colon A_\partial\to A_\partial\otimes_kH$ defined by $\Delta_\partial = \id\otimes \Delta$. The category ${}_{A_\partial}{\mathbf {Mod}}$ of graded left $A_\partial$-modules is isomorphic to 
$\pDGMod$. The $H$-comodule structure of $A_{\partial}$ plays no role in defining the category ${}_{A_\partial}{\mathbf {Mod}}$, but it is used to define the functor
\[\otimes \colon {}_{A_\partial}{\mathbf {Mod}} \times {}_{H}{\mathbf {Mod}}\to {}_{A_\partial}{\mathbf {Mod}}
\]
by $(M,C)\mapsto M\otimes_k C$ with $A_{\partial}$-action
\[A_{\partial}\otimes_k M \otimes_k C \stackrel{\Delta_\partial\otimes \id}{\longrightarrow} A_\partial\otimes_k H \otimes_k M\otimes_k C \cong A_\partial \otimes_k M \otimes_k H\otimes_k C \longrightarrow M\otimes_k C.
\]

 A morphism $f\colon M\to M'$ of $A_\partial$-modules is said to be \emph{null-homotopic} if it factors as a composite of the form $M\to N\otimes_kH \to M'$ for some $A_\partial$-module $N$. 
In fact it suffices that $f$ factors as 
\[
\xymatrix{
m\ar@{|->}[dr] & M \ar[dr]\ar[rr] && M' \\
& m\otimes d^{p-1} & M\otimes_k H[p-1] \ar[ur] & ,
}
\]
where $H[p-1]$ is the shift of $H$ with $H[p-1]_i=H_{i-p+1}$, see \cite[Lemma 1]{khovanov}.

A map $g\colon M\otimes_kH[p-1] \to M'$ is equivalent to $A$-module maps $g_i\colon M[p-1-i]\to M'$ from shifts $M[p-1-i]$ of $M$ such that $g_{i+1} +g_i\circ d =d\circ g_i$ for $0\leq i< p-1$ and $g_{p-1}\circ d=d\circ g_{p-1}$. Thus $g$ is completely determined by $g_0$. Moreover
\[ g_{p-1} =\sum_{i=0}^{p-1} (-1)^{p-1+i} \binom{p-1}{i} d^i g_0 d^{p-1-i} =\sum_{i=0}^{p-1} d^i g_0 d^{p-1-i}
\]
so that giving the map $g$ which factors the  map $f\colon M\to M'$ through $M\to M\otimes_{k}H[p-1]$, is equivalent to specifying a homotopy from $f$ to the zero map.

The associated stable module category 
${}_{A_\partial}\underline{{\mathbf {Mod}}}$ is defined as the quotient of 
${}_{A_\partial}{\mathbf {Mod}}$ by the ideal of null-homotopic maps. It agrees with the homotopy category of $\pDGMod$, i.e., the localization of $\pDGMod$ with respect to the homotopy equivalences.

By \cite[Theorem 1]{khovanov}, the stable module category ${}_{A_\partial}\underline{{\mathbf {Mod}}}$ is triangulated. The shift functor $T$ sends a $p$-DG module $M$ to the quotient in the short exact sequence
\[0\to M\to M\otimes_k H[p-1] \to TM \to 0.
\]

A triangle
\[X \to Y \to Z \to TX
\]
in ${}_{A_\partial}\underline{{\mathbf {Mod}}}$ is distinguished if it is isomorphic to a triangle of the form
\[M \stackrel{f}{\longrightarrow} N \longrightarrow \Cone(f) \longrightarrow TM,
\]
where $\Cone(f)$ denotes the pushout of
\[N\stackrel{f}{\longleftarrow} M \longrightarrow M\otimes_k H[p-1]
\]
in $\pDGMod$.

We call $\Cone(f)$ the \emph{mapping cone} of $f$. Explicitly, it is given by the graded $A$-module

\begin{equation}\label{eq:cone}
\Cone(f)= M[p-1]\oplus M[p-2]\oplus \ldots \oplus M[1] \oplus N
\end{equation}
with $p$-differential represented by
\[
\begin{pmatrix}
d & 0 & \ldots & & 0 & 0 \\
1 & d & \ldots && 0& 0 \\
0& 1 & & & 0 & 0\\
\vdots & &\ddots   && \vdots& \vdots \\
0 &       &&1 & d  & 0\\
0 & \ldots && 0 & f & d
\end{pmatrix}.
\]

By definition of $\Cone(f)$, we have a map of short exact sequences
\begin{equation}
\xymatrix{ 0\ar[r] & M \ar[r]\ar[d]^f & M\otimes_k H[p-1]\ar[d] \ar[r] & TM \ar[r]\ar[d]^\id & 0 \\
0\ar[r] &N \ar[r]&\Cone(f) \ar[r] & TM \ar[r] & 0.
}
\end{equation}

From the resulting map between the induced long exact sequences in Lemma \ref{lem:longexact} and the five lemma, it follows that $f\colon M\to N$ is a quasi-isomorphism if and only if $\Cone(f)$ is acyclic. The analogous characterization of homotopy equivalences holds as well. It is a direct consequence of the triangulated structure.

\begin{prop}\label{prop:cone}
A morphism $f$ of $p$-DG $A$-modules is a homotopy equivalence if and only if its mapping cone $\Cone(f)$ is contractible, i.e., chain homotopy equivalent to $0$.
\end{prop}

We say that a $p$-differential graded $A$-module $M$ is \emph{free} if the underlying graded $A$-module is free, i.e., if the underlying $A$-module is a direct sum of shifts of $A$. Similarly $M$ is said to be \emph{finitely generated} if the underlying graded $A$-module is so.

We identify $k$ with the quotient field $A/(x_1,\ldots, x_n)$. We will use repeatedly that for free, finitely generated graded $A$-modules, the functor $-\otimes_A k$ reflects isomorphisms. More generally, finitely generated can be replaced by the condition that the graded modules are zero in large enough degrees. The point is that the graded version of Nakayama's lemma holds for such modules.

\begin{lem}\cite[see e.g.~\S (I) Proposition~1]{carlsson83}\label{lem:reflectiso}
Let $f\colon M\to N$ be a morphism of free graded $A$-modules such that $M_n=N_n=0$ for all large enough $n$. If $f\otimes_A k$ is surjective, then $f$ is a retraction. If $f\otimes_A k$ is injective, then $f$ is a section.
\end{lem}

The next two results extend \cite[\S(I) Lemma 7]{carlsson83} and \cite[\S(I) Corollary 8]{carlsson83} from DG modules to $p$-DG modules. 

\begin{lem}
\label{lem:contractible}
Let $M$ be a free $p$-DG $A$-module such that $M_n=0$ for all large enough $n$. If ${_s}H_\bullet(M\otimes_A k)=0$ for some $1\leq s\leq p-1$, then $M$ is contractible, i.e., chain homotopy equivalent to $0$.
\end{lem}
\begin{proof}
Proposition \ref{prop:structureNcomplexes} provides a decomposition of any $p$-complex of vector spaces. Since ${_s}H_\bullet(M\otimes_A k)=0$,  the decomposition of the $p$-complex $M\otimes_A k$ is of the following form. There exist homogeneous elements $v_\alpha\in M\otimes_A k$ such that
\[M\otimes_A k\cong \bigoplus_\alpha kv_\alpha \oplus kd(v_\alpha)\oplus\ldots\oplus kd^{p-1}(v_\alpha).
\]
Let $F$ be the free graded $A$-module on generators $\{w_\alpha^0,\ldots,w_\alpha^{p-1}\}_\alpha$ with $w_\alpha^i$ in the same degree as $d^iv_\alpha$. The differential
\[d(w^i_\alpha)=\begin{cases}
w^{i+1}_\alpha \quad & i<p-1,\\
0 \quad & i=p-1,
\end{cases}
\]
equips $F$ with the structure of a $p$-DG $A$-module. For each $\alpha$, choose a lift $\overline{v}_\alpha\in M$ of $v_\alpha$. Then $F\to M$, $w^i_\alpha\mapsto d^i\overline{v}_\alpha$, defines a map of free $p$-DG $A$-modules. Moreover, the induced map $F\otimes_A k\to M\otimes_A k$ is an isomorphism. Thus $F\to M$ is itself an isomorphism by Lemma \ref{lem:reflectiso}. Finally, the identity map on $F$ is nullhomotopic via
\[h(w^i_\alpha)=\begin{cases}
0 \quad & i<p-1,\\
w_\alpha^0 \quad & i=p-1.
\end{cases}
\]
Thus $M$ is contractible as well.
\end{proof}

\begin{cor}
\label{cor:reflectquasiiso}
Let $f\colon M\to N$ be a map of free $p$-DG $A$-modules such that $M_n=N_n=0$ for all large enough $n$. If $f\otimes_A k$ is a quasi-isomorphism, then $f$ is a chain homotopy equivalence.
\end{cor}
\begin{proof}
By Proposition \ref{prop:cone}, it suffices to show that the mapping cone $\Cone(f)$ of the map $f\colon M\to N$ is contractible. If $M$ and $N$ are free, then so is $\Cone(f)$. Moreover the assumption that $M_n=N_n=0$ for large enough $n$ implies that $\Cone(f)_n=0$ for large enough $n$. As $\Cone(f)\otimes_A k \cong \Cone(f\otimes_A k)$, it follows that if  $f\otimes_A k$ is a quasi-isomorphism, then $\Cone(f\otimes_A k)$, and hence $\Cone(f)\otimes_A k$ is acyclic. Thus $\Cone(f)$ is contractible by Lemma \ref{lem:contractible} as desired.
\end{proof}

\subsection{The functor $\beta$}\label{sec:functorbeta}
Let $p$ be a prime number and let $k$ be a field of characteristic $p$. Let $G=(\Z/p)^n$ be an elementary abelian $p$-group of rank $n$. We  establish Theorem \ref{thm:firstmain} connecting chain complexes over $k[G]$ to $p$-DG modules over the polynomial ring $A$. First, we associate to a chain complex $C$ a $p$-complex $\iota C$ with the same homology groups, up to grading shifts and deletion of trivial groups. Thereafter, we  construct a $p$-DG module $\beta(\iota C)$, and prove Theorem \ref{thm:firstmain}.

For a chain complex $C$, let $\iota C$ denote the $p$-complex obtained by adding $p-2$ identity morphisms to the modules in even degrees as follows (displayed below the module is the degree of that term of $\iota C$)
\[\ldots \stackrel{=}{\longrightarrow} \underset{2p}{C_4} \stackrel{d}{\longrightarrow} \underset{2p-1}{C_3} \stackrel{d}{\longrightarrow} \underset{2p-2}{C_2} \stackrel{=}{\longrightarrow} \ldots \stackrel{=}{\longrightarrow} \underset{p}{C_2} \stackrel{d}{\longrightarrow} \underset{p-1}{C_1} \stackrel{d}{\longrightarrow} \underset{p-2}{C_0} \stackrel{=}{\longrightarrow} \ldots \stackrel{=}{\longrightarrow} \underset{0}{C_0}\stackrel{d}{\longrightarrow} \ldots
\]
So the modules of odd degrees $2i-1$ of $C$ now lie in degree $pi-1$, that is, $(\iota C)_{pi-1}=C_{2i-1}$. The $p-1$ modules between $(\iota C)_{pi-1}$ and $(\iota C)_{p(i+1)-1}$ are given by $C_{2i}$.

A straightforward calculation shows that up to regrading and deletion of trivial homology groups, the homology ${_s}H_\bullet(\iota C)$ agrees with the homology $H_\bullet(C)$ for any $1\leq s\leq p-1$.
\begin{lem}\label{lem:samehomology} Let $C$ be a chain complex and let $1\leq s\leq p-1$. The homology groups $_sH_\bullet(\iota C)$ are given by
\begin{align*}
_sH_{pi-1}(\iota C)&=H_{2i-1}(C) \\
_sH_{pi-1+s}(\iota C)&=H_{2i}(C)
\end{align*} 
for $i\in \Z$ and are zero in the remaining degrees. Moreover, the induced maps $i_*$ (see section \ref{sec:Ncomplexes}) are isomorphisms
\[{_1}H_l(\iota C) \cong {_2}H_l(\iota C)\cong \ldots \cong {_{p-1}}H_l(\iota C)
\]
whenever $l$ is congruent to $p-1$ mod $p$, and the induced maps $d_*$ (see section \ref{sec:Ncomplexes}) are isomorphisms
\[{_{p-1}}H_{l}(\iota C)\cong \ldots \cong {_2}H_{l-p+3}(\iota C) \cong {_1}H_{l-p+2}(\iota C)
\]
for $l$ congruent to $p-2$ mod $p$.
\end{lem}

\begin{rem}
The choice of $\iota C$ fits well into the history of $p$-complexes. Namely, working with a coefficient group whose elements are of order $p$, Mayer associated to a finite simplicial complex $K$ a $p$-complex $C(K)$ and introduced the groups $_sH_\bullet(C(K))$ in \cite{mayer}. Spanier proved in \cite{spanier} that these groups are not new invariants, but agree with the ordinary homology groups of $K$ in that the non-trivial ones are $_sH_{pi-1}(C(K))=H_{2i-1}(K)$ and $_sH_{pi-1+s}(C(K))=H_{2i}(K)$ for $i\in \Z$.  
\end{rem}

In defining the functor $\beta$, we will use that the group ring $k[G]$ is isomorphic to a truncated polynomial ring via
\begin{equation}\label{eq:groupring=truncpol}
\frac{k[y_1,\ldots,y_n]}{(y_1^p,\ldots,y_n^p)}\cong k[G], \quad y_i\mapsto e_i-1,
\end{equation}
where $e_i=(0,\ldots,1,\ldots,0)\in (\Z/p)^n$.

We continue writing $A=k[x_1,\ldots,x_n]$ for the graded polynomial ring with $\deg(x_i)=-1$. Since $k$ is of characteristic $p$, we can fix $q=1$ as a distinguished primitive root of unity of $A$.

We define a functor from the category $\pCh(k[G])$ of (unbounded) $p$-complexes of $k[G]$-modules to $\pDGMod$. Let
\[\beta\colon \pCh(k[G])\to \pDGMod\]
be the functor that sends a $p$-complex $C$ to $C\otimes_{k} A$ as a graded $k$-vector space. The $A$-module structure is defined by right multiplication and the differential $d\colon \beta(C)\to \beta(C)$ is induced by
\[d(c\otimes_{k} f)=d(c)\otimes_{k} f + \sum_{i=1}^n y_i c\otimes_{k} x_i f
\]
for homogeneous elements $c\in C$ and $f\in A$. To verify that $d^p=0$, use that
\begin{equation}\label{eq:differentialbeta}
d^l(c\otimes_k f)=\sum_{j=0}^l \sum_{1\leq i_1^j,\ldots,i_j^j\leq n} \binom{l}{j} y_{i_1^j}\ldots y_{i_j^j} d^{l-j}c\otimes_k x_{i_1^j}\ldots x_{i_j^j} f
\end{equation}
by induction on $l\geq 1$.

As for $p$-DG modules, we define freeness and finite generation for $C\in \pCh(k[G])$ via the corresponding notions for the underlying graded $k[G]$-modules.

\begin{lem}
\label{lem:totallyfinite}
Let $C\in \pCh(k[G])$. If $C$ is free and finitely generated, then
\[\dim_{k} {_s}H_\bullet(\beta(C))<\infty
\]
for all $1\leq s\leq p-1$.
\end{lem}
\begin{proof}
First, suppose that $C$ is given by $k[G]$ concentrated in degree $0$. Let $N=p$. Ignoring the grading, we can identify $\beta(C)=k[G]\otimes_{k} A$ with the $N$-complex $K^N_\bullet(x_1,\ldots,x_n)$ introduced in Definition \ref{def:koszulncomplex} for the variables $x_i\in A$, $1\leq i\leq n$. Indeed, this follows immediately using the identification \eqref{eq:groupring=truncpol} of the group ring with the truncated polynomial ring and denoting the generators in $K^N_\bullet(x_i)$ by
\[\ldots \rightarrow 0\rightarrow A y_i^0\stackrel{x_i}{\rightarrow} A y_i\stackrel{x_i}{\rightarrow}\ldots \stackrel{x_i}{\rightarrow} A y_i^{p-1}\rightarrow 0\rightarrow\ldots
.\] Thus in this case, the $k$-vector space ${_s}H_\bullet(\beta(C))$ is finite dimensional by Proposition \ref{prop:NKoszulfinitelygen}.

The general case follows by induction on the number of generators using the long exact sequence from Lemma \ref{lem:longexact}.
\end{proof}

The composite $\beta\iota$ induces an embedding on derived categories from perfect chain complexes to $p$-DG $A$-modules with totally finite dimensional homologies. We postpone the proof of this structural result to section \ref{sec:embedding} as it requires Theorem \ref{thm:compositionseries} from section \ref{sec:matrices}.

We establish our first main result. The case $k=\F_p$ is Theorem \ref{thm:firstmain} stated in the introduction. 
\begin{thm}\label{thm:firstmaintext} Let $G=(\Z/p)^n$. Let $k$ be a field of characteristic $p$ and let $A=k[x_1,\ldots,x_n]$ be the polynomial ring graded by $\deg(x_i)=-1$. Let $b_n$ be the minimum of
\[\sum_i \dim_{k} {_1}H_i(M\otimes_A k),
\]
where $M$ ranges over the finitely generated, free $p$-DG $A$-modules with $_sH_\bullet(M\otimes_A k)\neq 0$ and $\dim_{k} {_s}H_\bullet(M)<\infty$ for all $1\leq s \leq p-1$. Then
\[\sum_i \dim_{k} H_i(C)\geq b_n
\]
for all non-acyclic, perfect chain complexes $C$ over $k[G]$.
\end{thm}

\begin{proof}
Let $C$ be a non-acyclic, perfect chain complex over $k[G]$. Then $\iota C$ is a free and finitely generated $p$-complex over $k[G]$. Set $M=\beta(\iota C)$. By construction $M$ is free and finitely generated. We have $\dim_k {_s}H_\bullet(M) <\infty$ for all $1\leq s\leq p-1$ by Lemma \ref{lem:totallyfinite}. Note that $\beta(\iota C)\otimes_A k\cong \iota C$ as $p$-complexes. Thus
\[\dim_k {_1}H_\bullet(M\otimes_A k)=\dim_k {_1}H_\bullet(\iota C)=\dim_k H_\bullet(C),\]
where the last equality holds by Lemma \ref{lem:samehomology}. In particular ${_1}H(M\otimes_A k)\neq 0$ and hence ${_s}H(M\otimes_A k)\neq 0$  for all $1\leq s\leq p-1$ by Proposition \ref{prop:acyclic}. It follows that $\dim_k H_\bullet(C)=\dim_k {_1}H_\bullet(M\otimes_A k) \geq b_n$.
\end{proof}

In particular for $k=\F_p$, bounding $b_n\geq 2^n$ would establish Conjecture \ref{conjalg} on the total rank of the homology of perfect complexes over $\F_p[G]$. For $p=2$, Carlsson showed in \cite[Proof of Proposition II.9]{carlsson86} that Conjecture \ref{conjalg} implies $b_n\geq 2^n$ as well.

\begin{rem}
\label{rem:reductiontoclosed}
We can reduce to algebraically closed fields in the following sense. If $b_n$ is defined as in Theorem \ref{thm:firstmaintext} for $k$ and $b'_n$ is the analogous minimum for the algebraic closure $\overline{k}$, then $b_n\geq b'_n$. Indeed, let $M$ be a finitely generated, free $p$-DG $A$-module. Denote $A'=\overline{k}[x_1,\ldots,x_n]$. Then $M'=M\otimes_A A'$ is a finitely generated, free $p$-DG $A'$-module. Moreover $\dim_k {_s}H_\bullet(M)=\dim_{\overline{k}} {_s}H_\bullet(M')$ and
\[\dim_{k} {_s}H_\bullet(M\otimes_A k)=\dim_{\overline{k}} {_s}H_\bullet(M'\otimes_{A'} \overline{k})\]
for any $1\leq s \leq N-1$.

Alternatively, we can first replace the non-acyclic, perfect chain complex $C$ over $k[G]$ with $C'=C\otimes_{k[G]} \overline{k}[G]$ to obtain a perfect chain complex over $\overline{k}[G]$. Then $\dim_k H_\bullet(C)= \dim_{\overline{k}} H_\bullet(C')$ and thus $\dim_k H_\bullet(C)\geq b'_n$.
\end{rem}

\section{From $p$-DG modules over a polynomial ring to $p$-nilpotent matrices}\label{sec:matrices}
Let $p$ be a prime and let $k$ be a field of characteristic $p$. Let $A=k[x_1,\ldots, x_n]$ denote the polynomial ring in $n$ variables graded by $\deg(x_i)=-1$. If $M$ is a finitely generated, free $p$-DG $A$-module and we choose a basis, then the $p$-differential can be expressed by a $p$-nilpotent matrix $D=(f_{ij})$ with entries homogeneous polynomials. By $p$-nilpotent we mean that $D^p=0$. So for $p=2$ we  obtain square-zero matrices. In this case, Carlsson proved three results reducing Theorem \ref{thm:firstmaintext} to matrices of a particular form. The particular form of the matrices is expressed through conditions \ref{strictlyupper}, \ref{evalzero}, \ref{evalelse} in the following theorem.
\begin{thm}[Carlsson]\label{thm:p=2matrices}
Let $k$ be the algebraic closure of $\F_2$ and let $G=(\Z/2)^n$. Let $\ell_n$ be the minimum over all even integers $\ell>0$ for which there exist integers $c_1,\ldots,c_\ell$ and a square-zero $\ell\times \ell$-matrix $D=(f_{ij})$ with entries homogeneous polynomials $f_{ij}\in k[x_1,\ldots,x_n]$ of degree\footnote{Here the degree is taken with respect to the standard grading $\deg{x_i}=1.$} $c_i+1-c_j$ such that 
\begin{enumerate}[i)]
\item\label{strictlyupper}
$D$ is strictly upper triangular,
\item\label{evalzero}
$(f_{ij}(0))=0$, and
\item\label{evalelse}
the matrix $(f_{ij}(x))$ has rank $\ell/2$ for all $x\in k^n\setminus\{0\}$.
\end{enumerate}
Then $\dim_{\F_2} H_\bullet(C) \geq \ell_n$ for all non-acyclic, perfect chain complexes $C$ over $\F_2[G]$.
\end{thm}

Carlsson established Conjecture \ref{conjalg} for $p=2$ and $n=3$ by bounding $\ell_3\geq 2^3$ in Theorem \ref{thm:p=2matrices} above.

\begin{ex}\label{ex:examplep=2n=2}
Conjecture \ref{conjalg} for $p=2$ and $n=2$ follows easily from Theorem \ref{thm:p=2matrices}. Indeed, to bound $\ell_2\geq 2^2$, we have to exclude the existence of $2\times 2$-matrices $D=(f_{ij})$ satisfying the conditions above. Since $D$ is supposed to be strictly upper triangular, this matrix has only one potentially non-zero entry $f_{12}$. A non-constant polynomial in two variables over an infinite field has infinitely many zeros. Thus condition \ref{evalelse} implies that $f_{12}$ is constant and different from zero. This contradicts condition \ref{evalzero}.
\end{ex}
In this section, we will extend Theorem \ref{thm:p=2matrices} from $p=2$ to all primes. 
\subsection{Minimality}
The reason for condition \ref{evalzero} in Theorem \ref{thm:p=2matrices} is that any free, finitely generated DG $A$-module $M$ can be replaced by a chain homotopy equivalent one $\widetilde{M}$ which is \emph{minimal} in the sense that the differential of $\widetilde{M}\otimes_A k$ is zero. We extend this result, \cite[\S1 Proposition 7]{carlsson87}, to $p$-DG $A$-modules.

\begin{prop}
\label{prop:minimalmodel}
For any free, finitely generated $p$-DG $A$-module $M$, there exists a free, finitely generated $p$-DG $A$-module $\widetilde{M}$ and a homotopy equivalence $f\colon M\to \widetilde{M}$ such that the differential of $\widetilde{M}\otimes_A k$ is $(p-1)$-nilpotent.
\end{prop}
\begin{proof}
By Proposition \ref{prop:structureNcomplexes}, the $p$-complex $M\otimes_A k$ decomposes as 
\[M\otimes_A k \cong \bigoplus_\alpha kv_\alpha \oplus kd(v_\alpha)\oplus\ldots\oplus kd^{i_\alpha}(v_\alpha),
\]
for some homogeneous elements $v_\alpha$ of $M\otimes_A k$ with $i_\alpha=\max\{i; d^iv_\alpha\neq 0\}$.

Let $\overline{v}_\alpha\in M$ denote a lift of $v_\alpha$. Then $M$ is free on $\{\overline{v}_\alpha,d(\overline{v}_\alpha),\ldots,d^{i_\alpha}(\overline{v}_\alpha)\}_\alpha$ by Lemma \ref{lem:reflectiso}. The graded submodule of $M$ generated by all the elements $\{\overline{v}_\alpha,d(\overline{v}_\alpha),\ldots,d^{i_\alpha}(\overline{v}_\alpha)\}$ with $i_\alpha=p-1$ is a $p$-DG submodule. Let $\widetilde{M}$ denote the quotient of $M$ by this $p$-DG submodule and let  $f\colon M\to \widetilde{M}$ denote the quotient map. By construction $f\otimes_A k$ is a quasi-isomorphism and the differential of $\widetilde{M}\otimes_A k$ is $(p-1)$-nilpotent. We are left to show that $f$ is a homotopy equivalence. This follows immediately from Corollary \ref{cor:reflectquasiiso}.
\end{proof}

\subsection{Maximal rank}
For any $x\in k^n$, the image of the matrix $(f_{ij}(x))$ of Theorem \ref{thm:p=2matrices} is contained in its kernel since $(f_{ij})$ squares to zero. Thus the rank of $(f_{ij}(x))$ is at most $l/2$ by the rank-nullity theorem. Condition \ref{evalelse} states that $(f_{ij}(x))$ has maximal rank for $x\in k^n\setminus\{0\}$, or equivalently, that its image equals its kernel. This condition results from \cite[\S1 Proposition 8]{carlsson87}. We  extend \cite[\S1 Proposition 8]{carlsson87} to $p$-DG modules in Proposition \ref{prop:totallyfinite} below. There we  work with $N$-differential modules, i.e., modules $M$ together with an $N$-nilpotent endomorphism $d$ for some $N\geq 2$. For such a module $M$ and $1\leq s\leq N-1$ we write
\[_sH(M)=\ker d^s/\im d^{N-s}
.\] We say that $M$ is \emph{contractible} if $M$ is isomorphic to an $N$-differential module of the form $C\oplus\ldots \oplus C$ consisting of $N$ identical summands $C$ with differential
\[\begin{pmatrix}
0 & 1 & 0 &\ldots& 0 \\
 & 0 & 1 & & \\
 \vdots& & \ddots & \ddots & \\
  &&& 0 & 1 \\
 0 & &\ldots && 0
\end{pmatrix}.
\]

Before stating Proposition \ref{prop:totallyfinite}, we extend \cite[Remark 1.6]{avramovbuchweitziyengar} to $N$-complexes.
\begin{lem}
\label{lem:differentialmodulecontractible}
Let $R$ be a unitary ring and $P$ an $N$-differential $R$-module such that $_sH(P)=0$ for all $1\leq s\leq N-1$. If every $R$-module has finite projective dimension and $P$ is a projective $R$-module, then $P$ is contractible.
\end{lem}
\begin{proof}
Note that $\ker d\subset P$ is projective since $\ker d$ has finite projective dimension by assumption and we can build an exact sequence
\[0\longrightarrow \ker d \longrightarrow P \stackrel{d}{\longrightarrow} P \stackrel{d^{N-1}}\longrightarrow \ldots \stackrel{d}{\longrightarrow} P \longrightarrow \im d^{N-1} \longrightarrow 0
\]
of arbitrarily high length. We choose a splitting $s\colon \ker d\to P$ of the short exact sequence
\[0\longrightarrow \ker d^{N-1} \longrightarrow P \stackrel{d^{N-1}}{\longrightarrow} \ker d \longrightarrow 0.
\]
For any $1\leq i\leq N$, the submodule $\ker d^i$ decomposes as the internal direct sum $\ker d^{i-1} \oplus d^{N-i}s(\ker d)$. 
Thus $P=\ker d^N$ is the internal direct sum
\[P=d^{N-1}s(\ker d) \oplus d^{N-2}s(\ker d) \oplus \ldots \oplus s(\ker d),
\]
and $d$ maps $d^is (\ker d)$ isomorphically onto $d^{i+1} s(\ker d)$ for $0\leq i \leq N-2$ as desired.
\end{proof}

\begin{prop}
\label{prop:totallyfinite}
Let $M$ be a finitely generated free $p$-DG $A$-module. Let $m\neq (x_1,\ldots,x_n)$ be a maximal ideal in the ungraded polynomial ring $A$. If $\dim_k {_s}H_\bullet(M)<\infty$ for all $1\leq s\leq N-1$, then
\[
_sH(M\otimes_A A/m)=0
\]
for all $s$.
\end{prop}
\begin{proof}
Let $A_m$ denote $A$ localized at the maximal ideal $m$. First, we will prove that $_sH(M\otimes_A A_m)=0$ for all $s$. Since localization is exact, we have $_sH(M\otimes_A A_m)\cong {_s}H(M)\otimes_A A_m$. Any non-zero, homogeneous polynomial $f\in A$ acts nilpotently on $_sH_\bullet(M)$ as $_sH_\bullet(M)$ is finite dimensional over $k$ by assumption. Thus for any such $f$, the localization of $_sH(M)$ with respect to the powers of $f$ is zero: $_sH(M)\otimes_A A_f=0$. Choose a homogeneous polynomial $f$ that does not belong to $m$. Then
\[_sH(M)\otimes_A A_m\cong {_sH}(M)\otimes_A A_m \otimes_A A_f.\]
We conclude that $_sH(M\otimes_A A_m)=0$.

The ring $A_m$ has finite global dimension. Thus we can apply Lemma \ref{lem:differentialmodulecontractible} to conclude that $M\otimes_A A_m$ is a contractible $N$-differential $A_m$-module. Since
\[M\otimes_A A_m\otimes _{A_m} A/m\cong M\otimes_A A/m,
\]
it follows that $M\otimes_A A/m$ is contractible as well and in particular that its homology groups $_sH(M\otimes_A A/m)$ are zero for all $s$.
\end{proof}

\subsection{Composition series}
The triangular shape, condition \ref{strictlyupper} in Theorem \ref{thm:p=2matrices}, comes from the existence of composition series \cite[\S(I) Proposition 11]{carlsson83}. A composition series of a DG module $M$ is defined as a finite filtration of $M$ in which each successive quotient is free and has trivial differential. We extend this notion to $p$-DG modules as follows.

\begin{defn}
Let $M$ be a $p$-DG $A$-module. A \emph{composition series of $M$ of length $l$} is a filtration
\[0=M^0\subset M^1\subset \ldots \subset M^l=M
\]
such that each successive quotient $M^i/M^{i-1}$ is a direct sum of shifts of $A$ or of shifts of $A v\oplus A dv\oplus \ldots \oplus A d^{p-2} v$.
\end{defn}
We do not include more general quotients $Av\oplus Adv\oplus \ldots \oplus A d^jv$ for $0<j<p-2$ as they do not appear in our application.

By \cite[\S(I) Proposition 11]{carlsson83} any finitely generated, free DG $A$-module $M$ is chain homotopy equivalent to one that admits a composition series. To establish this result, Carlsson needed to pull back homology classes from $H_\bullet(M\otimes_A k)$ to $H_\bullet(M)$. For that, he considered the spectral sequence coming from the filtration $\{I^iM\}_i$ of $M$ for the augmentation ideal $I=(x_1,\ldots,x_n)$. The first differential in this spectral sequence is expressed with operations
\[\theta_j\colon H_\bullet(M\otimes_A k)\to H_\bullet(M\otimes_A k),\quad 0\leq j\leq n.\]
In the case of interest, $M=\beta(C)$ for a perfect chain complex $C$, these operations recover the action of the group ring $k[G]$ on \[H_\bullet(\beta(C)\otimes_A k)\cong H_\bullet(C),\] 
in that $\theta_j$ is multiplication by the element $y_j$ of the group ring under the identification \eqref{eq:groupring=truncpol}.

For $p>2$, we do not know how to extend these operations to an arbitrary finitely generated, free $p$-DG $A$-module, but it will be enough to establish them for the $p$-DG $A$-modules of interest.

Let $I\subset A$ be the graded ideal $(x_1,\ldots,x_n)$. Let $M$ be a free, finitely generated $p$-DG $A$-module. Filter $M$ as
\[\ldots \rightarrow I^{i+1} M\rightarrow I^i M \rightarrow \ldots \rightarrow I M \rightarrow M.
\]
Note that
\[\frac{I^i }{I^{i+1}} \otimes_A\frac{ M}{IM} \cong \frac{I^iM}{I^{i+1}M}, \quad [f]\otimes_A [m] \mapsto [f m]
\]
is an isomorphism of $p$-DG $A$-modules. As $I$ annihilates both arguments in the tensor product above, the tensor product can be taken over the field $A/I$ instead and we obtain an isomorphism on homology
\begin{equation}\label{eq:isohomology}
\frac{I^i }{I^{i+1}} \otimes_A {_s}H_{\bullet+i}\left(\frac{ M}{IM}\right)\cong  {_s}H_\bullet\left(\frac{I^iM}{I^{i+1}M}\right).
\end{equation}
Fix $l\in \Z$ and $1\leq s \leq p-1$. The filtration of $M$ above yields a filtration of chain complexes

\[\xymatrix{&\ar[d] &\ar[d] & &\ar[d]\\
\ldots \ar[r]& (I^{i+1}M)_l\ar[r]\ar[d]^{d^s} &(I^iM)_l \ar[r]\ar[d]^{d^s}& \ldots \ar[r] &M_l\ar[d]\ar[d]^{d^s}\\
\ldots \ar[r]& (I^{i+1}M)_{l-s}\ar[r]\ar[d]^{d^{p-s}} &(I^iM)_{l-s} \ar[r]\ar[d]^{d^{p-s}}& \ldots \ar[r] &M_{l-s}\ar[d]^{d^{p-s}}\\
\ldots \ar[r]& (I^{i+1}M)_{l-p}\ar[r]\ar[d] &(I^iM)_{l-p} \ar[r]\ar[d]& \ldots \ar[r] &M_{l-p}\ar[d]\\
& & & & .}
\]
The corresponding spectral sequence has $E^1$-page
\[\xymatrix{
{_{p-s}}H_{l-s}(I^2M/I^3M) & {_s}H_l(IM/I^2M)\ar[l]_{d^1} & {_{p-s}}H_{l+p-s}(M/IM)\ar[l]_{d^1} \\
_{s}H_{l-p}(I^2M/I^3M) & {_{p-s}}H_{l-s}(IM/I^2M)\ar[l]_{d^1} & {_{s}}H_{l}(M/IM)\ar[l]_{d^1} \\
{_{p-s}}H_{l-p-s}(I^2M/I^3M) & {_{s}}H_{l-p}(IM/I^2M)\ar[l]_{d^1} & {_{p-s}}H_{l-s}(M/IM)\ar[l]_{d^1} .}
\]
For all integers $j$, there exists $i$ such that $(I^iM)_j=0$ since $M$ is finitely generated and bounded above. Thus this spectral sequence converges to the homology of
\[\ldots \longrightarrow M_l \stackrel{d^s}{\longrightarrow} M_{l-s} \stackrel{d^{p-s}}{\longrightarrow} M_{l-p} \longrightarrow\ldots .
\]

Note that the page $E^\infty$ at the spot of ${_s}H_l(M/IM)$ is isomorphic to the image of
\[{_s}H_l(M)\rightarrow {_s}H_l(M/IM).
\]
The quotient $M/IM$ is just $M\otimes_A k$. We will need to pull back homology classes along the map ${_s}H_l(M)\rightarrow {_s}H_l(M\otimes_A k)$. Therefore, we study the differential
\begin{align*} d^1\colon &{_s}H_l(M/IM)\rightarrow {_{p-s}}H_{l-s}(IM/I^2M)\\
& [m+IM]\mapsto [d^sm+I^2M].
\end{align*}

Under the identification \eqref{eq:isohomology}, this differential becomes a map
\[ d^1\colon {_s}H_l(M/IM)\rightarrow I/I^2\otimes_A{_{p-s}}H_{l-s+1}(M/IM).\]

The finitely generated, free $p$-DG modules $M$ we are interested in, will satisfy the following assumption.
\begin{assumption}\label{ass:1} The induced maps $i_*$ (see section \ref{sec:Ncomplexes}) are isomorphisms
\[{_1}H_l(M/IM) \cong {_2}H_l(M/IM)\cong \ldots \cong {_{p-1}}H_l(M/IM)
\]
whenever $l$ is congruent to $p-1$ mod $p$, and the induced maps $d_*$ (see section \ref{sec:Ncomplexes}) are isomorphisms
\[{_{p-1}}H_{l}(M/IM)\cong \ldots \cong {_2}H_{l-p+3}(M/IM) \cong {_1}H_{l-p+2}(M/IM)
\]
for $l$ congruent to $p-2$ mod $p$.
In addition, $k[G]$ acts on $_sH_l(M/IM)$ for all $1\leq s \leq p-1$ and all $l\in \Z$ such that $i_*$ and $d_*$ are equivariant, and
\[{_1}H_{l}(M/IM)\stackrel{d^1}{\rightarrow} I/I^2 \otimes_A {_{p-1}}H_{l}(M/IM) \cong I/I^2 \otimes_A {_1}H_{l}(M/IM)
\]
is $[m]\mapsto \sum_i x_i\otimes_A y_i[m]$ for $l$ congruent to $p-1$ mod $p$
and
\[{_{p-1}}H_{l}(M/IM)\stackrel{d^1}{\rightarrow} I/I^2\otimes_A {_1}H_{l-p+2}(M/IM)\cong I/I^2\otimes_A{_{p-1}}H_l(M/IM)
\]
is $[m]\mapsto \sum_i x_i\otimes_A (p-1)y_i[m]$ for $l$ congruent to $p-2$ mod $p$.
\end{assumption}

\begin{lem}\label{lem:assumption} Let $C$ be a perfect complex over $k[G]$. Then the $p$-DG $A$-module $\beta(\iota C)$ satisfies Assumption \ref{ass:1}.
\end{lem}
\begin{proof}
As $\beta(\iota C)\otimes_A k\cong \iota C$, we find that $k[G]$ acts on $_sH_l(\beta(\iota C)\otimes_A k)$. Moreover $i_*$ and $d_*$ are equivariant isomorphisms in the desired degrees by Lemma \ref{lem:samehomology}. A straightforward calculation shows that
\[d^1\colon {_1}H_{l}(\iota C) \rightarrow I/I^2 \otimes_A {_1}H_{l}(\iota C)
\]
sends a homology class $[c]$ to $\sum_i x_i\otimes_A [y_i c]$ for $l$ congruent to $p-1$ mod $p$. Using \eqref{eq:differentialbeta}, one calculates that
\[d^1\colon {_{p-1}}H_{l}(\iota C)\rightarrow I/I^2\otimes_A{_{p-1}}H_l(\iota C)
\]
sends a homology class $[c]$ to $\sum_i x_i \otimes_A [(p-1)y_i c]$ for $l$ congruent to $p-2$ mod $p$.
\end{proof} 
  
The following result allows us to pull back homology classes.
\begin{lem}\label{lem:pullbackclasses}
Let $M$ be a free, finitely generated $p$-DG $A$-module satisfying Assumption \ref{ass:1}. Suppose $_sH_\bullet(M\otimes_A k)\neq 0$ for all $1\leq s\leq p-1$ and let $L$ be the largest dimension for which there exists an $s$ with ${_s}H_L(M\otimes_A k)\neq 0$. If $\xi\in {_s}H_L(M\otimes_A k)$ such that $y_i \xi=0$ for all $1\leq i\leq n$, then $\xi$ lies in the image of
\[{_s}H_L(M)\rightarrow {_s}H_L(M\otimes_A k).
\]
\end{lem}
\begin{proof}
From our assumptions on $i_*$ and $d_*$, it follows that $L$ is either congruent to $p-1$ mod $p$ and we can take $s=1$ or $L$ is congruent to $p-2$ mod $p$ and $s=p-1$. It is enough to show that $\xi$ is an infinite cycle in the spectral sequence above. In either case $d^1(\xi)=0$ by assumption and the higher differentials of $\xi$ vanish since the targets of these differentials are zero by the choice of $L$.
\end{proof}

We will use the lemma above and the following lemma to construct composition series.
\begin{lem}\label{lem:compseries}
Let
\[0\longrightarrow M \stackrel{i}{\longrightarrow} N \stackrel{p}{\longrightarrow} P \longrightarrow 0\]
be a short exact sequence of free, finitely generated $p$-DG $A$-modules. Suppose that $M$ is a sum of shifts of $A$ or a sum of shifts of $A v\oplus A dv\oplus \ldots \oplus A d^{p-2} v$, and that $P$ is chain equivalent to some finitely generated, free $p$-DG $A$-module that admits a composition series of length $l$. Then $N$ is chain equivalent to a finitely generated, free $p$-DG $A$-module $\overline{N}$ that admits a composition series of length $l+1$.
\end{lem}
\begin{proof}
Let $\overline{P}\to P$ be a chain equivalence from a finitely generated, free $p$-DG $A$-module that admits a composition series
\[0=\overline{P}^0\subset \overline{P}^1\subset \ldots \subset \overline{P}^l=\overline{P}.
\]
Let $\overline{N}$ be the pullback
\[\xymatrix{\overline{N}\ar[r]\ar[d] & \overline{P}\ar[d] \\
N\ar[r]_p & P
.}
\]
We will write $N=M\tilde{\oplus}P$ to indicate that $N$ is just the direct sum of $M$ and $P$ as graded $A$-modules when forgetting the differential. Similarly, we write $\overline{N}=M\tilde{\oplus}\overline{P}$. By construction, we have a map of short exact sequences
\[\xymatrix{0\ar[r] & M\ar[d]^=\ar[r] & M\tilde{\oplus}\overline{P}\ar[r]\ar[d] &\overline{P}\ar[d]\ar[r] & 0\\
0\ar[r] & M\ar[r] &M\tilde{\oplus}P \ar[r] & P \ar[r] & 0.}
\]
Note that $\overline{N}$ is finitely generated, free and chain equivalent to $N$. Taking further pullbacks, we obtain maps of short exact sequences with middle row
\[0\rightarrow M\tilde{\oplus} \overline{P}^0\rightarrow M\tilde{\oplus} \overline{P}^1\rightarrow \ldots \rightarrow M\tilde{\oplus} \overline{P}^l.
\]
This is a compositions series of $\overline{N}$ of length $l+1$.
\end{proof}
\begin{rem}\label{rem:(p-1)-nilpotentlem}
Let $N$, $\overline{N}$ and $\overline{P}$ be as in the proof of Lemma \ref{lem:compseries}. If the differentials of $N\otimes_A k$ and of $\overline{P}\otimes_A k$ are $(p-1)$-nilpotent, then the differential of $\overline{N}\otimes_A k$ is $(p-1)$-nilpotent as well.
\end{rem}
The following theorem allows us to replace the $p$-DG module $\beta(\iota C)$ coming from a perfect chain complex $C$ over $k[G]$ with a chain equivalent one that admits a composition series. The length of the composition series will be the sum of the Loewy lengths of the homology modules of $C$. Recall that the \emph{Loewy length} $\lol_{k[G]} M$ of a $k[G]$-module $M$ is
\[\lol_{k[G]} M=\inf \{l\geq 0| J^l M=0\},
\]
where $J$ denotes the augmentation ideal $J=(y_1,\ldots,y_n)$ of $k[G]$.
\begin{thm}\label{thm:compositionseries} Let $M$ be a free, finitely generated $p$-DG $A$-module. If $M$ satisfies Assumption \ref{ass:1}, then $M$ is chain equivalent to a free, finitely generated $p$-DG $A$-module that admits a composition series of length
\[\sum_i \lol_{k[G]} {_1}H_i(M\otimes_A k).\]
\end{thm}
\begin{proof}
Denote $\mathcal{L}(M)=\sum_i \lol_{k[G]} {_1}H_i(M\otimes_A k)$. If $\mathcal{L}(M)=0$, then $M$ is chain equivalent to $0$. Suppose that $\mathcal{L}(M)=j+1$ and the assertion of Theorem \ref{thm:compositionseries} holds for all $M'$ with $\mathcal{L}(M')\leq j$. Let $L$ be the largest dimension for which there exists an $s$ with ${_s}H_L(M\otimes_A k)\neq 0$ and set $\lambda=\lol_{k[G]} {_s}H_L(M\otimes_A k)$. Let $\zeta_1,\ldots,\zeta_r$ be a basis for $J^{\lambda-1}{_s}H_L(M\otimes_A k)$, where $J=(y_1,\ldots,y_n)$. By Lemma \ref{lem:pullbackclasses}, there exist representatives $\overline{\zeta}_1,\ldots,\overline{\zeta}_r\in M_L$ of homology classes in ${_s}H_L(M)$ such that the class $[\overline{\zeta}_i]$ is sent to $\zeta_i\in {_s}H_L(M\otimes_A k)$. If $L$ is congruent to $p-1$ mod $p$, we can assume that $s=1$. Define the map of $p$-DG $A$-modules
\[\varphi\colon \bigoplus_{i=1}^r Av_i \to M, \quad v_i\mapsto \overline{\zeta}_i.\]
Since $\varphi\otimes_A k$ is injective, the map $\varphi$ is an injection on a direct summand as a map of graded $A$-modules. Let $\overline{M}=M/\im \varphi$. Then $_	tH_i(\overline{M}\otimes_A k)\cong {_t}H_i(M\otimes_A k)$ for $i\neq L$ and all $t$, and 
\[_tH_L(\overline{M}\otimes_A k)\cong {_t}H_L(M\otimes_A k)/J^{\lambda-1}{_t}H_L(M\otimes_A k)
\]
for all $t$. So $\mathcal{L}(\overline{M})=\mathcal{L}(M)-1=j$. Moreover $\overline{M}$ satisfies Assumption \ref{ass:1} as well and therefore $\overline M$ is chain equivalent to a $p$-DG $A$-module that admits a composition series. We conclude this case by applying Lemma \ref{lem:compseries}.

If $L$ is congruent to $p-2$ mod $p$ and thus $s=p-1$, define
\[\varphi\colon \bigoplus_{i=1}^r Av_i \oplus Ad(v_i)\oplus \ldots \oplus Ad^{p-2}(v_i)\to M, \quad v_i\mapsto \overline{\zeta}_i\]
and proceed analogously to the above.
\end{proof}

\begin{rem}\label{rem:(p-1)-nilpotent}
It follows from Remark \ref{rem:(p-1)-nilpotentlem} that if for $M$ in Theorem \ref{thm:compositionseries} the differential of $M\otimes_A k$ is $(p-1)$-nilpotent, then the differential of $\overline{M}\otimes_A k$ for the constructed chain homotopy equivalent $p$-DG $A$-module $\overline{M}$ admitting a composition series is $(p-1)$-nilpotent as well.
\end{rem}

We assemble the results of this section to establish our main result. That is Theorem \ref{thm:secondmain} from the introduction extending Theorem \ref{thm:p=2matrices} from $p=2$ to all primes.

\begin{thm}\label{thm:secondmaintext} Let $G=(\Z/p)^n$. Let $k$ be a field of characteristic $p$ with algebraic closure $\overline{k}$. Let $\ell_n$ be the minimum over all multiples $\ell>0$ of $p$ for which there exist integers $c_1,\ldots,c_\ell$ and a $p$-nilpotent $\ell\times \ell$-matrix $D=(f_{ij})$ with entries homogeneous polynomials $f_{ij}\in \overline{k}[x_1,\ldots,x_n]$ of degree\footnote{Here the degree is taken with respect to the standard grading $\deg{x_i}=1.$} $c_i+1-c_j$ such that 
\begin{enumerate}[i)]
\item\label{pstrictlyupper}
$D$ is strictly upper triangular,
\item\label{pevalzero}
$(f_{ij}(0))^{p-1}=0$, and
\item\label{pevalelse}
the matrix $(f_{ij}(x))$ has rank $(p-1)l/p $ for all $x\in \overline{k}^n\setminus\{0\}$.
\end{enumerate}
Then $\dim_{k} H_\bullet(C) \geq 2\ell_n/p$ for all non-acyclic, perfect chain complexes $C$ over $k[G]$.
\end{thm}

\begin{proof}
Let $C$ be a non-acyclic, perfect chain complex $C$ over $k[G]$. Replacing $C$ by $C\otimes_{k[G]} \overline{k}[G]$ as in Remark \ref{rem:reductiontoclosed}, we may assume that $k$ is already algebraically closed. Let $M=\beta(\iota C)$. As established in the proof of Theorem \ref{thm:firstmaintext}, the $p$-DG $A$-module $M$ is free and finitely generated with $\dim_{k} {_s}H_\bullet(M)<\infty$ for $1\leq s\leq p-1$ and 
\[\dim_k {_1}H_\bullet(M\otimes_A k) =\dim_k H_\bullet(C).\]

By Proposition \ref{prop:minimalmodel}, we may assume that the $p$-differential of $M\otimes_A k$ is $(p-1)$-nilpotent. Lemma \ref{lem:assumption} ensures that Theorem \ref{thm:compositionseries} applies to $M$. By Theorem \ref{thm:compositionseries} and Remark \ref{rem:(p-1)-nilpotent} we can thus assume additionally that $M$ admits a composition series. As a graded $A$-module, $M$ is a direct sum of the successive quotients of the composition series. Let $\{v_1,\ldots, v_\ell\}$ be a basis of $M$ consisting of bases for the successive quotients. For the quotients consisting of terms of the form $Av\oplus Adv\oplus \ldots \oplus Ad^{p-2} v$, we pick $\{d^{p-2}v,\ldots,dv,v\}$ to compose the basis. With respect to the so chosen basis $\{v_1,\ldots, v_\ell\}$, the differential of the $p$-DG $A$-module $M$ can be expressed by a strictly upper triangular $\ell\times \ell$-matrix $(f_{i,j})_{i,j}$ consisting of homogeneous polynomials $f_{i,j}$. 
Denoting $c_i=\deg(v_i)$, the degree of $f_{i,j}$ in the standard grading of $k[x_1,\ldots,x_n]$ is $c_i+1-c_j$. The assumption that the $p$-differential of $M\otimes_A k$ is $(p-1)$-nilpotent implies condition \ref{pevalzero} $(f_{ij}(0))^{p-1}=0$. By Proposition \ref{prop:totallyfinite}, the $p$-differential $k$-module represented by $(f_{ij}(x))$ has trivial homology for all $x\neq 0$. This property is equivalent to condition \ref{pevalelse} $\rank (f_{ij}(x))= (\ell/p)(p-1)$. In particular, $p$ divides $\ell$. 

To establish $\dim_k H_\bullet(C)\geq 2\ell_n/p$, we will relate $\dim_k H_\bullet(C)$ to $\ell$. Since $(f_{ij}(0))$ is $(p-1)$-nilpotent,
\[\dim_k H_\bullet(C)=\dim_k {_1}H_\bullet(M\otimes_A k)=\dim_k \ker (f_{ij}(0)).\]

The Jordan normal form of the nilpotent matrix $(f_{ij}(0))$ is obtained from the Jordan normal form of the $p$-differential of $\iota C$ by eliminating the Jordan blocks of size $p$. That is, it consists of $\sum_i \dim_k H_{2i}(C)$ Jordan blocks

\[\begin{pmatrix}
0 & 1 & 0 &\ldots& 0 \\
 & 0 & 1 & & \\
 \vdots& & \ddots & \ddots & \\
  &&& 0 & 1 \\
 0 & &\ldots && 0
\end{pmatrix}
\]
of size $(p-1)$ and of $\sum_i \dim_k H_{2i+1}(C)$ Jordan blocks $\begin{pmatrix}
0 
\end{pmatrix}$
of size $1$. In particular, $\ell= (p-1)(\sum_i \dim_k H_{2i}(C)) +\sum_i \dim_k H_{2i+1}(C)$. Let $\chi(C)$ denote the Euler characteristic of $C$. If $\chi(C)=0$, then $\ell=(\dim_k H_\bullet(C))p/2$ and thus $\dim_k H_\bullet(C)\geq 2\ell_n/p$ as desired.

Since $C$ is a perfect complex, the order of the group $G$ divides the Euler characteristic $\chi(C)$. Thus if $\chi(C)\neq 0$, then $\dim_k H_\bullet(C)\geq p^n$. We conclude by showing that $p^n\geq 2 \ell_n/p$. Considering $(\Z/p)^n$ as a subgroup of the $n$-torus $(S^1)^n$, the $n$-torus becomes a free $G$-CW complex. Its cellular chain complex has Euler characteristic zero and the total rank of its homology is $2^n$.  Therefore, the matrix $(f_{ij})$ associated as above to this chain complex has size $l=2^{n-1}p$. It follows that $p^n\geq 2^{n}\geq 2\ell_n/p$ as desired.
\end{proof}

As in Example \ref{ex:examplep=2n=2} for $p=2$, we immediately recover the known result that Conjecture \ref{conjalg} holds for $n\leq 2$.
\begin{cor}\label{cor:lowdimcases}
Let $G=(\Z/p)^n$ be an elementary abelian $p$-group of rank $n\leq 2$. Let $C$ be a perfect chain complex over $\F_p[G]$. If $C$ is not acyclic, then
\[\dim_{\F_p} H_\bullet(C)\geq 2^n.
\]
\end{cor}
\begin{proof}
We apply Theorem \ref{thm:secondmaintext}. Since $\dim_{\F_p} H_\bullet(C)\geq 2\ell_n/p$, it suffices to show that $\ell_n\geq 2^{n-1}p$. By definition of $\ell_n$ in Theorem \ref{thm:secondmaintext}, we know that $\ell_n$ is a positive multiple of $p$. In particular for $n=1$, we have $\ell_1\geq p$ as desired. If $n=2$, we want to exclude matrices $D=(f_{ij})$ as in Theorem \ref{thm:secondmaintext} of size $\ell=p$. Since $D$ is supposed to be strictly upper triangular, the matrix $D^{p-1}$ has one potentially non-zero entry. Namely the entry in the top right corner
 \[(D^{p-1})_{1p}=f_{12}f_{23}\ldots f_{p-1,p}.\]

For any $x\neq 0$, condition \ref{pevalelse} implies that the Jordan type of the matrix $(f_{ij}(x))$ consists of one Jordan block of length $p$. Therefore $f_{12}(x)\ldots f_{p-1,p}(x)$ is non-zero for any $x\neq 0$. It follows that $f_{12}\ldots f_{p-1,p}$ is constant and different from zero. This contradicts condition \ref{pevalzero}.
\end{proof}

It will be interesting to see if Conjecture \ref{conjalg} can be established in the first open case, $n=3$ for odd primes, using Theorem \ref{thm:secondmaintext}.

\section{Embedding of derived categories}\label{sec:embedding}
Let $k$ be a field of characteristic $p>0$. Let $G=(\Z/p)^n$, $A=k[x_1,\ldots,x_n]$ with $\deg(x_i)=-1$. Let $\Perf(k[G])$ denote the category of perfect $k[G]$-chain complexes. If $p=2$, the composite $\beta\iota$ from section \ref{sec:functorbeta} induces an equivalence from the derived category $\D(\Perf(k[G]))$ to the derived category of finitely generated, free DG $A$-modules with finite dimensional total homology. For arbitrary $p$, we prove the postponed structural result that $\beta\iota$ induces an embedding. The proof uses a hopfological algebra result, \cite[Corollary~6.10]{qi} of Qi, which identifies the hom-set in the derived category from a $p$-DG module satisfying property (P) as introduced in \cite[Definition~6.3]{qi} with the hom-set in the homotopy category.

\begin{thm}\label{thm:embedding}
The composite $\beta\iota$ induces an embedding of derived categories
\[\beta\iota\colon \D(\Perf(k[G]))\to \D(\pDGMod).\]
\end{thm}

\begin{proof}
The functor $\iota$ preserves quasi-isomorphisms by Lemma \ref{lem:samehomology}. The functor $\beta$ preserves quasi-isomorphisms $f$ between finitely generated $p$-complexes by Corollary \ref{cor:reflectquasiiso} as $\beta(f)\otimes_A k\cong f$. Thus the composite $\beta\iota\colon \Perf(k[G])\to \pDGMod$ preserves quasi-isomorphisms as well and therefore induces a functor on derived categories.

To show that the induced functor $\beta\iota$ is an embedding, we identify the hom-sets as follows. For perfect chain complexes $C$ and $D$ the set of morphisms $\D(\Perf(k[G]))(C,D)$ agrees with the set of morphisms in the homotopy category from $C$ to $D$. The analogous result holds for the hom-set between $\beta\iota C$ and $\beta\iota D$:
\[\D(\pDGMod)(\beta\iota C, \beta\iota D) \cong{}_{A_\partial}\underline{{\mathbf {Mod}}}(\beta\iota C, \beta\iota D)
\]
Indeed, the $p$-DG $A$-module $\beta\iota C$ admits a composition series by Theorem \ref{thm:compositionseries} and Lemma \ref{lem:assumption}. In particular $\beta\iota C$ satisfies property (P) from \cite[Definition~6.3]{qi}, and therefore the identification of hom-sets above holds by \cite[Corollary~6.10]{qi}. 

We show that $\beta\iota$ is faithful. Let $f\colon C\to D$ be a map in $\Perf(k[G])$. If $\beta\iota f$ is null-homotopic, then $\beta\iota f\otimes_A k$ is null-homotopic as well. Since $\beta\iota f\otimes_A k\cong \iota f$, it suffices to show that $\iota$ reflects null-homotopic maps. If $\{h_i\}_i$ is a null-homotopy for $\iota f$, then
\[
\begin{cases} & d(h_{pl+p-2}+\ldots h_{pl+1})+h_{pl}\colon C_{2l}\to C_{2l+1} \\
& h_{pl-1}\colon C_{2l-1}\to C_{2l}
\end{cases}
\]
for $l\in \Z$ is a null-homotopy for $f$.

By construction $\beta\iota$ is injective on objects. Thus $\beta\iota\colon \D(\Perf(k[G]))\to \D(\pDGMod)$ is an embedding.
\end{proof}
\providecommand{\bysame}{\leavevmode\hbox to3em{\hrulefill}\thinspace}
\providecommand{\MR}{\relax\ifhmode\unskip\space\fi MR }
\providecommand{\MRhref}[2]{%
  \href{http://www.ams.org/mathscinet-getitem?mr=#1}{#2}
}
\providecommand{\href}[2]{#2}

\vspace{20pt}
\scriptsize
\noindent
Jeremiah Heller\\
University of Illinois at Urbana-Champaign\\
Department of Mathematics\\
1409 W. Green Street, Urbana, IL 61801\\
\texttt{jbheller@illinois.edu}

\vspace{10pt}
\noindent
Marc Stephan\\
Lehrstuhl f\"ur Differentialgeometrie\\
Universit\"at Augsburg\\
86135 Augsburg\\
Germany\\
\texttt{marc.stephan@math.uni-augsburg.de}

\begin{thebibliography}{10}

\bibitem{ademswan}
A.~Adem and R.~G. Swan, \emph{Linear maps over abelian group algebras}, J. Pure
  Appl. Algebra \textbf{104} (1995), no.~1, 1--7. \MR{1359686}

\bibitem{alldaypuppe}
C.~Allday and V.~Puppe, \emph{Cohomological methods in transformation groups},
  Cambridge Studies in Advanced Mathematics, vol.~32, Cambridge University
  Press, Cambridge, 1993. \MR{1236839 (94g:55009)}

\bibitem{avramovbuchweitziyengar}
L.~L. Avramov, R.~Buchweitz, and S.~Iyengar, \emph{Class and rank of
  differential modules}, Invent. Math. \textbf{169} (2007), no.~1, 1--35.
  \MR{2308849 (2008h:13032)}

\bibitem{avramovbuchweitziyengarmiller}
L.~L. Avramov, R.~Buchweitz, S.~Iyengar, and C.~Miller, \emph{Homology of
  perfect complexes}, Adv. Math. \textbf{223} (2010), no.~5, 1731--1781.
  \MR{2592508}

\bibitem{baumgartner}
C.~Baumgartner, \emph{On the cohomology of free {$p$}-torus actions}, Arch.
  Math. (Basel) \textbf{61} (1993), no.~2, 137--149. \MR{1230942 (95k:57045)}

\bibitem{Benson-Pevtsova}
D.~Benson and J.~Pevtsova, \emph{A realization theorem for modules of
  constant {J}ordan type and vector bundles}, Trans. Amer. Math. Soc.
  \textbf{364} (2012), no.~12, 6459--6478. \MR{2958943}

\bibitem{BensonBook}
D.~Benson, \emph{Representations of elementary abelian {$p$}-groups and
  vector bundles}, Cambridge Tracts in Mathematics, vol. 208, Cambridge
  University Press, Cambridge, 2017. \MR{3585474}

\bibitem{bichon}
J.~Bichon, \emph{{$N$}-complexes et alg\`ebres de {H}opf}, C. R. Math. Acad.
  Sci. Paris \textbf{337} (2003), no.~7, 441--444. \MR{2023750}

\bibitem{carlsson83}
G.~Carlsson, \emph{On the homology of finite free {$({\bf
  Z}/2)^{n}$}-complexes}, Invent. Math. \textbf{74} (1983), no.~1, 139--147.
  \MR{722729 (85h:57045)}

\bibitem{carlsson86}
\bysame, \emph{Free {$({\bf Z}/2)^k$}-actions and a problem in commutative
  algebra}, Transformation groups, {P}ozna\'n 1985, Lecture Notes in Math.,
  vol. 1217, Springer, Berlin, 1986, pp.~79--83. \MR{874170 (88g:57042)}

\bibitem{carlsson87}
\bysame, \emph{Free {$({\bf Z}/2)^3$}-actions on finite complexes}, Algebraic
  topology and algebraic {$K$}-theory ({P}rinceton, {N}.{J}., 1983), Ann. of
  Math. Stud., vol. 113, Princeton Univ. Press, Princeton, NJ, 1987,
  pp.~332--344. \MR{921483 (89g:57054)}

\bibitem{dubois-violette}
M.~Dubois-Violette, \emph{Tensor product of {$N$}-complexes and generalization
  of graded differential algebras}, Bulg. J. Phys. \textbf{36} (2009), no.~3,
  227--236. \MR{2640833}

\bibitem{dubois-violettekerner}
M.~Dubois-Violette and R.~Kerner, \emph{Universal {$q$}-differential calculus
  and {$q$}-analog of homological algebra}, Acta Math. Univ. Comenian. (N.S.)
  \textbf{65} (1996), no.~2, 175--188. \MR{1451169}

\bibitem{Friedlander-Pevtsova}
E.~M.~Friedlander and J.~Pevtsova, \emph{Constructions for infinitesimal
  group schemes}, Trans. Amer. Math. Soc. \textbf{363} (2011), no.~11,
  6007--6061. \MR{2817418}

\bibitem{iyamakatomiyachi}
O.~Iyama, K.~Kato, and J.~Miyachi, \emph{Derived categories of
  {$N$}-complexes}, J. Lond. Math. Soc. (2) \textbf{96} (2017), no.~3,
  687--716. \MR{3742439}

\bibitem{iyengarwalker}
S.~B. {Iyengar} and M.~E. {Walker}, \emph{Examples of finite free complexes of
  small rank and small homology}, Acta Math.
\textbf{221} (2018), no.~1, 143--158.  \MR{3877020}

\bibitem{kapranov}
M.~M. {Kapranov}, \emph{On the q-analog of homological algebra},
  \url{http://arxiv.org/abs/q-alg/9611005v1}, 1991.

\bibitem{kasselwambst}
C.~Kassel and M.~Wambst, \emph{Alg\`ebre homologique des {$N$}-complexes et
  homologie de {H}ochschild aux racines de l'unit\'e}, Publ. Res. Inst. Math.
  Sci. \textbf{34} (1998), no.~2, 91--114. \MR{1617062}

\bibitem{khovanov}
M.~Khovanov, \emph{Hopfological algebra and categorification at a root of
  unity: the first steps}, J. Knot Theory Ramifications \textbf{25} (2016),
  no.~3, 1640006, 26. \MR{3475073}

\bibitem{khovanovqi}
M.~Khovanov and Y.~Qi, \emph{An approach to categorification of some small
  quantum groups}, Quantum Topol. \textbf{6} (2015), no.~2, 185--311.
  \MR{3354331}

\bibitem{mayer}
W.~Mayer, \emph{A new homology theory. {I}, {II}}, Ann. of Math. (2)
  \textbf{43} (1942), 370--380, 594--605. \MR{0006514}

\bibitem{qi}
Y.~Qi, \emph{Hopfological algebra}, Compos. Math. \textbf{150} (2014), no.~1,
  1--45. \MR{3164358}

\bibitem{spanier}
E.~Spanier, \emph{The {M}ayer homology theory}, Bull. Amer. Math. Soc.
  \textbf{55} (1949), 102--112. \MR{0029169}

\bibitem{tikaradze}
A.~Tikaradze, \emph{Homological constructions on {$N$}-complexes}, J. Pure
  Appl. Algebra \textbf{176} (2002), no.~2-3, 213--222. \MR{1933716}

\end{thebibliography}
\end{document}